\documentclass[11pt, reqno]{amsart}

\usepackage[top=3.75cm, bottom=3cm, left=3cm, right=3cm]{geometry}
\usepackage{amsmath}
\usepackage{amsthm}
\usepackage{amsfonts}
\usepackage{amssymb}
\usepackage{eucal}
\usepackage{bm}
\usepackage[colorlinks=true, citecolor=cyan, urlcolor=cyan, linkcolor=magenta]{hyperref}
\usepackage{hyperref}
\usepackage{color}
\usepackage[all,cmtip]{xy}
\usepackage{enumitem}
\usepackage{tikz}  
\usepackage{graphicx}
\usepackage{scalerel}
\usepackage{comment}
\usepackage{mathtools} 

\def\citestacks#1{\cite[\href{https://stacks.math.columbia.edu/tag/#1}{Tag #1}]{stacks-project}}

\numberwithin{equation}{section}

\theoremstyle{plain}
\newtheorem{theorem}{Theorem}[section]
\newtheorem{lemma}[theorem]{Lemma}
\newtheorem{proposition}[theorem]{Proposition}
\newtheorem{corollary}[theorem]{Corollary}
\newtheorem{conjecture}[theorem]{Conjecture}

\theoremstyle{definition}
\newtheorem{definition}[theorem]{Definition}
\newtheorem{example}[theorem]{Example}
\newtheorem{remark}[theorem]{Remark}

\makeatletter
\newtheoremstyle{italicsname}
 {3pt}
 {3pt}
 {\itshape}
 {}
 {\itshape}
 {.}
 {.5em}
 {\thmname{#1}\thmnumber{\@ifnotempty{#1}{ }#2}%
 \thmnote{ {\the\thm@notefont(#3)}}}
\makeatother
\theoremstyle{italicsname}
\newtheorem{innerstep}{Step}
\newenvironment{step}[1]
 {\renewcommand\theinnerstep{#1}\innerstep}
 {\endinnerstep}

\setlist[itemize]{leftmargin=*, itemsep={2pt}}
\setlist[enumerate]{leftmargin=*, itemsep={2pt}} 

\newcommand{\st}{\mid}

\newcommand{\llangle}{\left \langle}
\newcommand{\rrangle}{\right \rangle}
\DeclareMathOperator{\Dperf}{D_{{perf}}}
\DeclareMathOperator{\Dqc}{D_{qc}}
\newcommand{\Cat}{\mathrm{Cat}}
\newcommand{\PrCat}{\mathrm{PrCat}}
\newcommand{\Fun}{\mathrm{Fun}}
\DeclareMathOperator{\Ind}{Ind}

\DeclareMathOperator{\Map}{Map}

\newcommand{\cB}{\mathcal{B}}

\newcommand{\Ku}{\CMcal{K}u}

\DeclareMathOperator{\Stab}{Stab}

\newcommand{\Gr}{\mathrm{Gr}}

\DeclareMathOperator{\Spec}{Spec}

\newcommand{\wtilde}{\widetilde}

\newcommand{\Pf}{\mathrm{Pf}}

\newcommand{\cHom}{\mathcal{H}\!{\it om}}
\DeclareMathOperator{\Hom}{Hom}
\DeclareMathOperator{\Ext}{Ext}
\DeclareMathOperator{\cExt}{\mathcal{E}\!{\it xt}}
\newcommand{\Tor}{\mathrm{Tor}}
\DeclareMathOperator{\Cone}{Cone}

\newcommand{\svee}{\scriptscriptstyle\vee}
\newcommand{\id}{\mathrm{id}}

\newcommand{\tH}{\wtilde{\mathrm{H}}}

\newcommand{\vphi}{\varphi}

\newcommand{\Ktop}[1][]{\rK_{#1}^{\rtop}}
\DeclareMathOperator{\Sp}{Sp}
\newcommand{\rtop}{\mathrm{top}}
\DeclareMathOperator{\HH}{HH}
\DeclareMathOperator{\HN}{HN}
\DeclareMathOperator{\HP}{HP}
\DeclareMathOperator{\Hdg}{Hdg}
\DeclareMathOperator{\CH}{CH}

\DeclareMathOperator{\Br}{Br}

\DeclareMathOperator{\gr}{gr}

\newcommand{\an}{\mathrm{an}}

\newcommand{\Shv}{\mathrm{Shv}}

\newcommand{\one}{\mathbf{1}}

\newcommand{\coev}{\mathrm{coev}}
\newcommand{\ev}{\mathrm{ev}}

\newcommand{\op}{\mathrm{op}}
\newcommand{\dR}{\mathrm{dR}}

\DeclareMathOperator{\Tr}{Tr}

\newcommand{\Art}{\mathrm{Art}}
\newcommand{\Sets}{\mathrm{Sets}}

\DeclareMathOperator{\Def}{Def}

\newcommand{\coker}{\mathrm{coker}}

\newcommand{\Voi}{\mathrm{V}}

\newcommand{\tors}{\mathrm{tors}}
\newcommand{\tf}{\mathrm{tf}}

\DeclareMathOperator{\ch}{ch} 
\DeclareMathOperator{\td}{td} 

\newcommand{\cO}{\mathcal{O}}
\newcommand{\cA}{\mathcal{A}}
\newcommand{\cC}{\mathcal{C}}
\newcommand{\cD}{\mathcal{D}}

\newcommand{\cM}{\CMcal{M}}

\newcommand{\cU}{\mathcal{U}}

\newcommand{\rK}{\mathrm{K}}
\newcommand{\rH}{\mathrm{H}}

\newcommand{\rh}{\mathrm{h}}
\newcommand{\rS}{\mathrm{S}}

\newcommand{\rM}{\mathrm{M}}

\newcommand{\rR}{\mathrm{R}}

\newcommand{\rT}{\mathrm{T}}

\newcommand{\bC}{\mathbf{C}}

\newcommand{\bG}{\mathbf{G}}

\newcommand{\bZ}{\mathbf{Z}}
\newcommand{\bP}{\mathbf{P}}
\newcommand{\bQ}{\mathbf{Q}}

\begin{document}

\title[The integral Hodge conjecture for CY2 categories]{The integral Hodge conjecture for two-dimensional Calabi--Yau categories}

\author{Alexander Perry}
\address{Department of Mathematics, University of Michigan, Ann Arbor, MI 48109 \smallskip}
\email{arper@umich.edu}

\thanks{This work was partially supported by NSF grant DMS-1902060 and the Institute for Advanced Study.}

\begin{abstract}
We formulate a version of the integral Hodge conjecture for categories, prove the conjecture for two-dimensional Calabi--Yau categories which are suitably deformation equivalent to the derived category of a K3 or abelian surface, and use this to deduce cases of the usual integral Hodge conjecture for varieties. 
Along the way, we prove a version of the variational 
integral Hodge conjecture for families of two-dimensional Calabi--Yau categories, as well as a general smoothness result for relative moduli spaces of objects in such families. 
Our machinery also has applications to the structure of intermediate Jacobians, 
such as a criterion in terms of derived categories for when they split as a sum of Jacobians of curves. 
\end{abstract}

\maketitle

\section{Introduction} 
\label{section-introduction} 

Let $X$ be a smooth projective complex variety. 
The Hodge conjecture in degree $n$ for $X$ states that the subspace of Hodge classes in $\rH^{2n}(X, \bQ)$ 
is generated over $\bQ$ by the classes of algebraic cycles of codimension $n$ on $X$. 
This conjecture holds for $n = 0$ and $n = \dim(X)$ for trivial reasons,  
for $n=1$ by the Lefschetz $(1,1)$ theorem, and 
for $n = \dim(X) - 1$ by the case $n = 1$ and the hard Lefschetz theorem. 
In all other degrees, the conjecture is far from being known in general, and is one of the deepest open problems in algebraic geometry. 

There is an integral refinement of the conjecture, which is in 
fact the version originally proposed by Hodge \cite{hodge}. 
Let $\Hdg^n(X, \bZ) \subset \rH^{2n}(X, \bZ)$ denote the subgroup of integral Hodge classes, consisting of cohomology classes whose image in $\rH^{2n}(X, \bC)$ is of type $(n,n)$ for the Hodge decomposition. 
Then the cycle class map $\CH^n(X) \to \rH^{2n}(X, \bZ)$ factors through $\Hdg^n(X, \bZ)$.  
The integral Hodge conjecture in degree $n$ states that the image of this map is precisely $\Hdg^n(X, \bZ)$. 
This implies the rational version from above, 
and is known for $n = 0, 1, \dim(X)$ for the same reasons. 
However, in all other degrees, the integral Hodge conjecture is false in general. 
Indeed, Atiyah and Hirzebruch constructed the first of many counterexamples 
\cite{atiyah-hirzebruch-hodge, trento, soule-voisin, CT-voisin, totaro-IHC-number-field, schreieder2, benoist-ottem} 
showing that Hodge's original hope is quite far from being true. 

The failure of the integral Hodge conjecture is measured by the cokernel $\Voi^n(X)$ 
of the map $\CH^n(X) \to \Hdg^n(X, \bZ)$, 
which we call the degree $n$ \emph{Voisin group} of $X$. 
This is a finitely generated abelian group, predicted to be finite by the Hodge conjecture. 
The group $\Voi^n(X)$ is especially interesting for $n = 2$ or $n = \dim(X) - 1$ 
because then it is birationally invariant, as observed by Voisin \cite{soule-voisin}. 
In particular, for rational varieties $\Voi^n(X)$ vanishes in these degrees, i.e. the integral Hodge 
conjecture holds. 
This is the first in a line of results which show that, despite the counterexamples mentioned above, 
the integral Hodge conjecture may hold under interesting geometric conditions. 
For instance, for $n = 2$ the conjecture is known if $X$ is a threefold of negative Kodaira 
dimension or of Kodaira dimension zero with $\rH^0(X, K_X) \neq 0$ \cite{voisin-IHC-3fold, totaro-IHC-CY3}, 
a fibration in quadrics over a surface \cite{CT-voisin}, 
or a fibration in at worst nodal cubic threefolds over a curve \cite{voisin-AJ}. 
For $n = \dim(X) - 1$, the conjecture is known if $X$ is a Fano fourfold \cite{horing-voisin}, a Fano 
variety of index $\dim(X) - 3$ and $\dim(X) = 5$ or $\dim(X) \geq 8$ \cite{horing-voisin, floris}, 
or a hyperk\"{a}hler variety of K3$^{[n]}$ or generalized Kummer type \cite{ottem-HK}. 

We highlight two general questions suggested by these results: 
\begin{itemize}
\item When does the integral Hodge conjecture hold for varieties with $K_X = 0$? 
\item When does the integral Hodge conjecture hold in degree $2$ for Fano fourfolds? 
\end{itemize}
The second question is of particular importance because its failure obstructs rationality. 
The main goal of this paper is to give a positive answer to the 
first question for certain ``noncommutative surfaces'', and to use this to provide  
positive answers to the second question for interesting examples. 
To do so, we develop Hodge theory for suitable categories, 
introduce a technique involving moduli spaces of objects in categories to prove 
the integral Hodge conjecture, 
and prove a smoothness result for such moduli spaces in the two-dimensional Calabi--Yau case 
that also has applications to hyperk\"{a}hler geometry. 
Our Hodge-theoretic apparatus for categories can also be applied to 
questions about the odd degree cohomology of varieties, 
such as when an intermediate Jacobian splits as a sum of Jacobians of curves. 

\subsection{The integral Hodge conjecture for categories}
We will be concerned with an analogue of the above story for 
a ``noncommutative smooth proper complex variety'', i.e. an admissible subcategory 
$\cC \subset \Dperf(X)$ of the derived category of 
a smooth proper complex variety $X$. 
For any such $\cC$, we show that the 
(the zeroth homotopy group of) Blanc's topological K-theory \cite{blanc} gives a finitely generated abelian group 
$\Ktop[0](\cC)$ which is equipped with a canonical weight $0$ Hodge structure, whose Hodge decomposition is given in terms of Hochschild homology. 
Moreover, the natural map from the Grothendieck group $\rK_0(\cC) \to \Ktop[0](\cC)$ factors through 
the subgroup $\Hdg(\cC, \bZ) \subset \Ktop[0](\cC)$ of integral Hodge classes. 
The integral Hodge conjecture for $\cC$ then states that the 
map $\rK_0(\cC) \to \Hdg(\cC, \bZ)$ is surjective, while the Hodge conjecture for $\cC$ states that 
this is true after tensoring with $\bQ$. 

When $\cC = \Dperf(X)$, after tensoring with $\bQ$ the construction $\rK_0(\cC) \to \Hdg(\cC, \bZ)$ 
recovers the usual cycle class map 
$\CH^*(X) \otimes \bQ \to \Hdg^*(X, \bQ)$
to the group of rational Hodge classes of all degrees.  
Therefore, the Hodge conjecture in all degrees 
for $X$ is equivalent to the Hodge conjecture for $\Dperf(X)$. 
The integral Hodge conjectures for $X$ and $\Dperf(X)$ are more subtly, but still very closely, related 
(Proposition~\ref{proposition-IHC-vs-NCIHC}). 

The key motivating example for us is 
the \emph{Kuznetsov component} $\Ku(X) \subset \Dperf(X)$ of a cubic fourfold $X \subset \bP^5$, 
defined by the semiorthogonal decomposition 
\begin{equation*}
\Dperf(X) = \llangle \Ku(X) , \cO_X, \cO_X(1), \cO_X(2) \rrangle. 
\end{equation*} 
Kuznetsov \cite{kuznetsov-cubic} proved that $\Ku(X)$ is a \emph{two-dimensional Calabi--Yau (CY2) category}, i.e. $\Ku(X)$ satisfies Serre duality in the form 
\begin{equation*}
\Ext^i(E, F) \cong \Ext^{2-i}(F, E)^{\svee} \quad \text{for} \quad  E, F \in \Ku(X), 
\end{equation*} 
and is connected in the sense that its zeroth Hochschild cohomology is $1$-dimensional. 
The simplest example of a CY2 category is $\Dperf(T)$ where $T$ is a K3 or abelian surface, 
or more generally the twisted derived category $\Dperf(T, \alpha)$ for a Brauer class 
$\alpha \in \Br(T)$. 
Kuznetsov proved that for special $X$ 
the category $\Ku(X)$ is equivalent to such an example, 
while by \cite{addington-thomas} no such equivalence exists for very general $X$. 
Since then a number of further CY2 categories have been discovered 
(see \S\ref{section-CY2-examples}), the next most studied being the Kuznetsov 
component 
of a Gushel--Mukai fourfold (a Fano fourfold that generically can be written as the 
intersection of the Grassmannian $\Gr(2,5)$ with a hyperplane and a quadric). 
Recently, CY2 categories have attracted a great deal of attention due to their connections to 
birational geometry, Hodge theory, and the construction of hyperk\"{a}hler varieties 
\cite{addington-thomas, huybrechts-cubic, huybrechts-torelli, BLMS, LPZ, BLMNPS, 
KuzPerry:dercatGM, GMstability}. 

Inspired by work of Addington and Thomas~\cite{addington-thomas}, 
for any CY2 category $\cC$ we define the \emph{Mukai Hodge structure} $\tH(\cC, \bZ)$ as the 
weight $2$ Tate twist of $\Ktop[0](\cC)$. 
The group $\tH(\cC, \bZ)$ is also equipped with a natural pairing $(-,-)$, defined as the negative of the Euler pairing. 
In the case where $\cC = \Dperf(T)$ for a K3 or abelian surface $T$, this recovers the classical Mukai Hodge structure. 

Our first main result gives a criterion for the validity of the integral Hodge conjecture for a CY2 category. 
This criterion is of a variational nature, and depends on the notion of a family of CY2 categories. 
In general, the notion of a family of categories can be formalized as an $S$-linear 
admissible subcategory $\cC \subset \Dperf(X)$, where $X \to S$ is a morphism of varieties. 
There is a well-behaved notion of base change for such categories, 
which gives rise to a fiber category $\cC_s \subset \Dperf(X_s)$ for any point $s \in S$. 
When $X \to S$ is smooth and proper, we show that a relative version of topological K-theory 
from \cite{moulinos} gives a local system $\Ktop[0](\cC/S)$ on $S^{\an}$ 
underlying a canonical variation of Hodge 
structures of weight $0$, which fiberwise recovers the Hodge structure on $\Ktop[0](\cC_s)$ from above. 

We say $\cC \subset \Dperf(X)$ is a \emph{CY2 category over $S$} if $X \to S$ is smooth and proper 
and the fibers $\cC_s$ are CY2 categories.  
For example, if $X \to S$ is a family of cubic fourfolds, then 
similar to the case where the base is a point, one can define a CY2 category 
$\Ku(X) \subset \Dperf(X)$ over $S$ with fibers $\Ku(X)_s \simeq \Ku(X_s)$  
(Example~\ref{example-cubic}). 
As in the absolute case, we define the Mukai local system $\tH(\cC/S, \bZ)$ of a CY2 
category $\cC$ over $S$ as a Tate twist of $\Ktop[0](\cC/S)$. Now we can state our first main theorem. 

\begin{theorem}
\label{theorem-IHC} 
Let $\cC$ be a CY2 category over $\bC$. 
Let $v \in \Hdg(\cC, \bZ)$. 
Assume there exists a CY2 category $\cD$ over a complex variety 
$S$ with points $0, 1 \in S(\bC)$ such that: 
\begin{enumerate}
\item $\cD_0 \simeq \cC$. 
\item $\cD_1 \simeq \Dperf(T, \alpha)$ where $T$ is a K3 or abelian surface and $\alpha \in \Br(T)$ is a Brauer class. 
\item $v$ remains of Hodge type along $S$, i.e. extends to a section 
of the local system $\tH(\cD/S, \bZ)$. 
\end{enumerate}
Further, assume $(v,v) \geq -2$ or $(v,v) \geq 0$ according to whether 
$T$ is a K3 or abelian surface. 
Then $v$ is algebraic, i.e. lies in the image of $\rK_0(\cC) \to \tH(\cC, \bZ)$. 

In particular, if the cokernel of the map $\rK_0(\cC) \to \Hdg(\cC, \bZ)$ is generated  
by elements $v$ as above, then this map is in fact surjective, i.e. the integral Hodge conjecture holds for $\cC$. 
\end{theorem} 

In practice, this reduces the integral Hodge conjecture for a given CY2 category to 
checking that it deforms within any Hodge locus to a category of the form $\Dperf(T, \alpha)$ 
(see Remark~\ref{remark-CY2-IHC}). 
We apply the theorem to prove the integral Hodge conjecture for the Kuznetsov components of 
cubic and Gushel--Mukai fourfolds, and use this to deduce the following consequence. 

\begin{corollary}
\label{corollary-example-IHC}
The integral Hodge conjecture in degree $2$ holds for cubic fourfolds and Gushel--Mukai fourfolds. 
\end{corollary} 

This result is new for Gushel--Mukai fourfolds. For cubic fourfolds it was originally proved by Voisin 
\cite[Theorem 18]{Voisin-aspects-of-Hodge-conjecture},  
and was recently reproved in \cite{BLMNPS} using the 
construction of Bridgeland stability conditions on the Kuznetsov component and the theory of 
stability conditions in families. 
One of the main contributions of this paper is to show that 
the particular geometry of cubic fourfolds 
and the difficult ingredients about stability conditions can be excised from the proof of~\cite{BLMNPS}, 
giving a general tool for attacking cases of the integral Hodge conjecture. 

Corollary~\ref{corollary-example-IHC} is natural from the point of view of rationality problems. 
One of the biggest open conjectures in classical algebraic geometry is the irrationality of very general cubic fourfolds. 
The same conjecture for Gushel--Mukai fourfolds is closely related and expected to be equally difficult. 
Corollary~\ref{corollary-example-IHC} shows there is no obstruction to rationality for these fourfolds 
coming from the integral Hodge conjecture.  
Our argument applies more generally to any fourfold whose derived category decomposes into a collection of exceptional objects and 
a CY2 category that deforms within any Hodge locus to one of the form $\Dperf(T, \alpha)$. 
This jibes with the fact that, despite many recent advances on the rationality problem 
\cite{voisin-universal2cycle, CTP, totaro-hypersurfaces, HPT, schreieder1, schreieder2, nicaise-shinder, kontsevich-tschinkel},  
irrationality results remain out of reach for such fourfolds. 

Our methods also 
lead to bounds on the torsion order of Voisin groups. 
As illustrations, we show that $\Voi^3(X)$ is $2$-torsion for $X$ a Gushel--Mukai sixfold (Corollary~\ref{corollary-GM6}), 
and that $\Voi^4(X)$ is $6$-torsion for $X \subset \bP^3 \times \bP^3 \times \bP^3$ a smooth $(1,1,1)$ divisor (Corollary~\ref{corollary-dumb-IHC}). 

\subsection{The variational integral Hodge conjecture and moduli spaces of objects} 
Now we explain the idea of the proof of Theorem~\ref{theorem-IHC}, which involves some results of independent interest. 
The first is an instance of the variational integral Hodge conjecture for categories.  
Recall that an object $E$ of the derived category of a variety is called 
simple if $\Hom(E, E)$ is $1$-dimensional and universally gluable if $\Ext^{< 0}(E, E) = 0$. 

\begin{theorem}
\label{theorem-VHC} 
Let $\cC$ be a CY2 category over a complex variety $S$. 
Let $\vphi$ be a section of the local system $\tH(\cC/S, \bZ)$. 
Assume there exists a complex point $0 \in S(\bC)$ such that the 
fiber $\vphi_0 \in \tH(\cC_0, \bZ)$ is the class of a simple universally gluable object of $\cC_0$.  
Then $\vphi_s \in \tH(\cC_s, \bZ)$ is algebraic
for every $s \in S(\bC)$, i.e. lies in the image of $\rK_0(\cC_s) \to \tH(\cC_s,\bZ)$. 
\end{theorem} 

This implies Theorem~\ref{theorem-IHC} because twisted derived categories of 
K3 or abelian surfaces always contain many simple universally gluable objects. 

Our proof of Theorem~\ref{theorem-VHC} relies on moduli spaces of objects in categories. 
For any $S$-linear admissible subcategory $\cC \subset \Dperf(X)$ where $X \to S$ is a smooth proper morphism of complex varieties, Lieblich's work \cite{lieblich} gives an 
algebraic stack $\cM(\cC/S) \to S$ parameterizing universally gluable objects in $\cC$. 
For any section $\vphi$ of the local system $\Ktop[0](\cC/S)$, there is an open substack 
$\cM(\cC/S, \vphi)$ parameterizing objects of class $\vphi$. 
We prove that if there is a point $0 \in S(\bC)$ such that the fiber $\vphi_0$ can be represented by 
the class of an object in $\cC_0$ at which the morphism $\cM(\cC/S, \vphi) \to S$ is smooth, 
then $\vphi_s$ is algebraic for every $s \in S(\bC)$ (Proposition~\ref{proposition-VHC-criterion}). 
This gives a general method for proving the variational Hodge conjecture for categories, 
which can be thought of as a noncommutative version of Bloch's method from~\cite{bloch-semiregularity}. 

In his seminal paper \cite{mukai-K3}, Mukai proved that the moduli space of simple sheaves on a K3 or abelian surface is smooth. 
More recently, Inaba generalized this to moduli spaces of objects in the derived category of such a surface \cite{inaba}. 
The following further generalization 
replaces a fixed surface with a family of CY2 categories, 
and implies Theorem~\ref{theorem-VHC}. 
We write $s\cM(\cC/S, \vphi) \subset \cM(\cC/S, \vphi)$ for the open substack of 
simple objects, which is a $\bG_m$-gerbe over an algebraic space $s\rM(\cC/S,\vphi)$ 
(Lemma~\ref{lemma-sM-lft}). 

\begin{theorem}
\label{theorem-mukai} 
Let $\cC$ be a CY2 category over a complex variety $S$. 
Let $\vphi$ be a section of the local system $\tH(\cC/S, \bZ)$ 
whose fibers $\vphi_s \in  \tH(\cC_s, \bZ)$ 
are Hodge classes for all $s \in S(\bC)$. 
Then $s\cM(\cC/S, \vphi)$ and $s\rM(\cC/S,\vphi)$ are smooth over $S$. 
\end{theorem}

In Theorem~\ref{theorem-mukai}, 
it is in fact enough to assume a single fiber of $\vphi$ is a Hodge class, because then all fibers are (Lemma~\ref{lemma-invariant-cycles}). 

\begin{remark}
Theorem~\ref{theorem-mukai} plays a crucial role in recent constructions of 
hyperk\"{a}hler varieties as moduli spaces of Bridgeland stable objects 
in CY2 categories \cite{BLMNPS, GMstability}. 
Namely, the theorem allows one to prove facts (e.g. nonemptiness) about such moduli spaces by 
deformation from a special CY2 category (e.g. the derived category of a K3 surface). 
The special case of Theorem~\ref{theorem-mukai} where $\cC$ is the Kuznetsov component of a 
family of cubic fourfolds was first proved in 
\cite[Theorem 3.1]{BLMNPS}, using properties of cubic fourfolds. 
Our result does not use anything about the ambient variety containing $\cC$ in its derived category, 
and thus provides a general tool for studying moduli spaces of objects in families of CY2 categories, 
which has already been put to use in \cite{GMstability}. 
\end{remark}

This paper's approach to the (variational) integral Hodge conjecture via moduli spaces of objects 
may be useful in other contexts. 
It would be interesting, for instance, to apply this method 
to varieties whose Kuznetsov components are not CY2 categories. 
In a different direction, we plan to develop in a sequel to this paper a 
version of our results in positive characteristic, with applications to the integral Tate conjecture.  

\subsection{Intermediate Jacobians} 
\label{section-intro-IJ}
The Hodge conjecture concerns the even degree cohomology of a variety, but there are also many interesting questions about the Hodge structures on odd degree cohomology. 
The machinery developed in this paper gives a version of Hodge theory in odd degree for noncommutative varieties. 
Namely, for an admissible subcategory $\cC \subset \Dperf(X)$ of the derived category of a smooth proper complex variety $X$ and any integer $n$, we show that the $n$-th homotopy group of Blanc's topological K-theory gives a finitely generated abelian group 
$\Ktop[n](\cC)$ which is equipped with a canonical weight $-n$ Hodge structure, whose Hodge decomposition is given in terms of Hochschild homology. 
These Hodge structures are Tate twists of each other for varying even or odd $n$, so there are essentially only two of interest --- $\Ktop[0](\cC)$ discussed above, and $\Ktop[1](\cC)$. 

When $\cC = \Dperf(X)$, the rational Hodge structure 
$\Ktop[1](\cC) \otimes \bQ$ recovers the rational odd cohomology $\rH^{\mathrm{odd}}(X, \bQ)$, with the weight $-1$ Hodge structure obtained by taking appropriate Tate twists in each degree. 
The integral relation is more subtle, but at least assuming that $X$ is odd-dimensional and its odd degree cohomology is concentrated in degree $\dim(X)$,  
$\Ktop[-\dim(X)](\Dperf(X))$ recovers the polarized Hodge structure 
$\rH^{\dim(X)}(X, \bZ)$ (Proposition~\ref{proposition-odd-cohomology}).

This has many applications to the structure of intermediate Jacobians. 
Recall that if $X$ is a smooth proper complex variety, 
then for any odd integer $k$ the intermediate Jacobian 
$J^k(X)$ is a complex torus constructed from the Hodge structure $\rH^k(X, \bZ)$, 
which is in fact a canonically principally polarized abelian variety if the Hodge decomposition 
only has two terms (i.e. $\rH^k(X, \bC) = \rH^{p,q}(X) \oplus \rH^{q,p}(X)$ for some $p,q$) and $k = \dim(X)$. 
Intermediate Jacobians have vast applications in algebraic geometry, 
ranging from irrationality results \cite{clemens-griffiths, beauville-IJ}, to Torelli theorems \cite{debarre-torelli-3-quads, voisin-quartic-double, debarre-quartic-double}, 
to infinite generation results for algebraic cycles 
\cite{clemens-griffiths-group, voisin-griffiths-group}. 
As a sample application of our techniques, we prove the following. 

\begin{theorem}
\label{theorem-IJ} 
Let $X$ be a smooth proper complex variety of odd dimension $n$, 
such that $\rH^k(X, \bZ) = 0$ for all odd $k < n$. 
Assume there is a semiorthogonal decomposition 
\begin{equation*}
\Dperf(X) = \llangle \Dperf(Y_1), \dots, \Dperf(Y_m) \rrangle 
\end{equation*} 
where each $Y_i$ is a smooth proper complex variety of dimension $n_i$, such that: 
\begin{itemize}
\item if $n_i$ is odd then $\rH^{k}(Y_i, \bZ) = 0$ for all odd $k < n_i$, and 
\item if $n_i$ is even then $\rH^{\mathrm{odd}}(Y_i, \bQ) = 0$. 
\end{itemize} 
Then there is an isomorphism of complex tori 
\begin{equation}
\label{J-isomorphism} 
J^n(X) \cong \bigoplus_{n_i \, \mathrm{odd}} J^{n_i}(Y_i) . 
\end{equation} 
If we further assume that there is a fixed integer $t \geq 0$ such that 
for all odd $n_i$ we have $\rH^{n_i}(Y_i, \bC) = \rH^{p_i, q_i}(Y_i) \oplus \rH^{q_i, p_i}(Y_i)$ where 
$p_i - q_i = 2t +1$, then 
$\rH^n(X, \bC) = \rH^{p,q}(X) \oplus \rH^{q,p}(X)$ where $p - q = 2t+1$  
and~\eqref{J-isomorphism} is an isomorphism of principally polarized abelian varieties. 
\end{theorem} 

We formulate the following special case of our results which is of particular interest. 

\begin{corollary}
\label{corollary-IJ-threefold} 
Let $X$ be a smooth proper complex threefold such 
that $\rH^1(X, \bZ) = 0$. 
Assume there is a semiorthogonal decomposition 
\begin{equation} 
\label{Db-rcX} 
\Dperf(X) = \llangle \Dperf(C_1), \dots, \Dperf(C_r), E_1, \dots, E_s \rrangle 
\end{equation} 
where each $C_i$ is a smooth proper curve and $E_j \in \Dperf(X)$ is an exceptional object. 
Then there is an isomorphism 
\begin{equation} 
\label{J3-rcX}
J^3(X) \cong J^1(C_1) \oplus \cdots \oplus J^1(C_n) 
\end{equation} 
of principally polarized abelian varieties. 
If we further assume that $\rH^5(X, \bZ) = 0$, then we have 
$\rH^3(X, \bZ)_{\tors} = 0$, i.e. $\rH^3(X, \bZ)$ is torsion free. 
\end{corollary} 

The first part of Corollary~\ref{corollary-IJ-threefold} is an immediate consequence of 
Theorem~\ref{theorem-IJ}, while we prove the second part in \S\ref{section-IJ} as a consequence 
of a result (Proposition~\ref{proposition-odd-cohomology}) relating odd cohomology 
to odd topological K-theory. 

\begin{remark} 
Let $X$ be a rationally connected smooth proper complex threefold. 
Then $\rH^1(X, \bZ) = \rH^5(X, \bZ) = 0$. 
Hence Corollary~\ref{corollary-IJ-threefold} shows the existence of a 
semiorthogonal decomposition of the form~\eqref{Db-rcX} implies that $X$ satisfies 
both the Clemens--Griffiths criterion for rationality (the splitting of the intermediate Jacobian as a sum of Jacobians of curves) \cite{clemens-griffiths} 
and the Artin--Mumford criterion (the vanishing of $\rH^3(X,\bZ)_{\tors}$) \cite{artin-mumford}. 
Kuznetsov's rationality conjectures \cite{kuznetsov-rationality, kuznetsov-cubic} (see also \cite{bernardara-representability}) predict that if $X$ is rational, then a semiorthogonal decomposition of the form~\eqref{Db-rcX} exists. 
Thus, our result shows that Kuznetsov's conjectural criterion implies the classical two. 
\end{remark} 

Theorem~\ref{theorem-IJ} and Corollary~\ref{corollary-IJ-threefold} 
(as well as Proposition~\ref{proposition-odd-cohomology} in the body of the paper)
greatly generalize many results in the literature relating intermediate Jacobians to derived categories \cite{bernardara-conic, bernardara-representability, bernardara-tabuada, kuznetsov-IJ}. 
For instance, the main result of \cite{bernardara-conic} is the splitting~\eqref{J3-rcX} in the very special case where $X$ is a standard conic bundle over a rational surface with a decomposition~\eqref{Db-rcX} that is suitably compatible with the conic bundle structure; 
similarly, \cite[Proposition~8.4]{kuznetsov-IJ} gives the splitting~\eqref{J3-rcX} in the case where $X$ is a rationally connected threefold and there is a single curve in the decomposition~\eqref{Db-rcX}. 

More generally, our results can be used to relate intermediate Jacobians of varieties whose derived 
categories have a semiorthogonal component in common. 
Bernardara and Tabuada \cite{bernardara-tabuada} previously studied this problem using 
noncommutative motives, but in general their results only give isogenies between the algebraic parts of intermediate Jacobians,  
which can only be shown to be isomorphisms under hypotheses that are difficult to check in practice. 
Our results, on the other hand, only require cohomological hypotheses which are easy to check, 
give simple proofs of the applications considered in~\cite{bernardara-tabuada}, and also apply to many cases 
inaccessible by the results there (see Example~\ref{example-XL-YL}). 

As a final application, we give a simple proof of a recent result of Debarre and Kuznetsov \cite{GM-IJ}, which identifies the intermediate Jacobians of odd-dimensional Gushel--Mukai varieties that are ``generalized partners or duals'' (Theorem~\ref{theorem-GM-IJ}). 
Answering a question of Kuznetsov (see \cite[Remark 1.2]{GM-IJ}), 
we show that this follows from the equivalence 
proved in \cite[Theorem~1.6]{categorical-cones} between the Kuznetsov 
components of such varieties. 

The results of this paper suggest developing other aspects of 
the Hodge theory of categories for applications to classical algebraic geometry. 
For instance, it would be interesting to study Abel--Jacobi maps taking values 
in the intermediate Jacobians of categories (as defined in Definition~\ref{definition-IJC}); 
we leave this to future investigation. 

\subsection{Organization of the paper} 
We begin by reviewing the framework of categories linear over a base scheme in \S\ref{section-linear-categories}. 
In \S\ref{section-HH-homology} and \S\ref{section-HH-cohomology} we review some aspects of 
Hochschild homology and cohomology, which are needed later in the paper 
for studying the Hodge theory of categories and the deformation theory of objects in a category. 
In \S\ref{section-hodge-theory} we develop the Hodge theory of categories; in particular, we formulate the 
(integral) Hodge conjecture for categories and relate it to the corresponding conjecture for varieties, 
as well as prove the results on intermediate Jacobians described above. 
In \S\ref{section-CY2} we define CY2 categories and their associated Mukai Hodge structures, 
and survey the known examples of CY2 categories. 
In \S\ref{section-moduli-objects} we prove Theorem~\ref{theorem-mukai} on the smoothness of relative moduli spaces of objects in families of CY2 categories. 
Finally, in \S\ref{section-theorem-proofs} we prove our other main results 
--- Theorem~\ref{theorem-IHC}, Corollary~\ref{corollary-example-IHC}, and Theorem~\ref{theorem-VHC} --- 
as well as several complementary results. 

\subsection{Conventions}
\label{section-conventions}
All schemes are assumed to be quasi-compact and quasi-separated. 
A variety over a field $k$ is an integral scheme which is separated and of finite type over $k$. 
For a scheme $X$, $\Dperf(X)$ denotes the category of perfect complexes and 
$\Dqc(X)$ denotes the unbounded derived category of quasi-coherent sheaves. 
If $\alpha \in \Br(X)$ is a Brauer class, $\Dperf(X, \alpha)$ denotes 
the category of perfect complexes over an Azumaya algebra $\CMcal{A}$ representing $\alpha$, 
consisting of complexes of $\CMcal{A}$-modules which are locally quasi-isomorphic to a bounded 
complex of locally projective $\CMcal{A}$-modules of finite rank. 
When $X$ is a smooth variety, as will be the case whenever $\Dperf(X, \alpha)$ is considered in this paper, this category 
agrees with the bounded derived category of coherent $\alpha$-twisted sheaves \cite[Lemma 10.19]{kuznetsov-hyperplane}. 
All functors are derived by convention. 
In particular for a morphism $f \colon X \to Y$ of 
schemes we write $f_*$ and $f^*$ for the derived pushforward and pullback functors, 
and for $E, F \in \Dperf(X)$ we write $E \otimes F$ for the derived tensor product. 
For technical convenience, all categories are regarded as $\infty$-categories as reviewed in \S\ref{section-linear-categories}, but most arguments in the paper can be made at the triangulated level for admissible subcategories of derived categories of varieties. 

\subsection{Acknowledgements} 
The genesis of this paper was the Workshop on Derived Categories, Moduli Spaces, and Deformation Theory at Cetraro in June 2019. 
There, Manfred Lehn asked me  
whether stability conditions could be removed from the proof of the integral Hodge conjecture for cubic fourfolds in \cite{BLMNPS}, 
and an ensuing discussion with Daniel Huybrechts convinced me that this was possible. 
I thank both of them for their role in inspiring this paper, 
as well as the organizers of the conference for creating such a stimulating environment.

I would also like to thank Nicolas Addington, Arend Bayer, Bhargav Bhatt, Olivier Debarre, 
Sasha Kuznetsov, Jacob Lurie, Emanuele Macr\`{i}, Tasos Moulinos, Laura Pertusi, and Xiaolei Zhao 
for discussions, comments, and questions about this work. 


\section{Linear categories}  
\label{section-linear-categories}

In this paper we use the formalism of categories linear over a base scheme. 
We summarize the key points of this theory here following \cite{NCHPD}, which is based on 
Lurie's work \cite{HA}. 
Throughout this section we fix a (quasi-compact and quasi-separated) base scheme $S$. 

\subsection{Small linear categories} 
\label{linear-categories} 
An \emph{$S$-linear category} $\cC$ is a small idempotent-complete 
stable $\infty$-category equipped with a module structure over $\Dperf(S)$. 
The collection of all $S$-linear categories is organized into an $\infty$-category $\Cat_S$, 
which admits a symmetric monoidal structure. 
For $\cC, \cD \in \Cat_S$ we denote by 
\begin{equation*}
\cC \otimes_{\Dperf(S)} \cD \in \Cat_S 
\end{equation*} 
their tensor product. 
A morphism $\cC \to \cD$ in $\Cat_S$, also called an \emph{$S$-linear functor}, is an exact functor that suitably commutes with the action of $\Dperf(S)$; 
these morphisms form the objects of an $S$-linear category $\Fun_S(\cC, \cD)$, which is the internal mapping object in the category $\Cat_S$. 

If $\cC \in \Cat_S$ and $T \to S$ is a morphism of schemes, then the tensor product 
\begin{equation*}
\cC_{T} = \cC \otimes_{\Dperf(S)} \Dperf(T)  \in \Cat_{T} 
\end{equation*} 
is naturally a $T$-linear category, called the \emph{base change} of $\cC$ along $T \to S$. 
If $s \in S$ is a point with residue field $\kappa(s)$, then we write $\cC_s$ for the $\kappa(s)$-linear 
category obtained by base change along $\Spec(\kappa(s)) \to S$, and call it the \emph{fiber} of $\cC$ over $s \in S$. 
In this way, an $S$-linear category $\cC$ can be thought of as a family of categories 
parameterized by $S$. 

\begin{example}
\label{example-Dperf-X}
Let $f \colon X \to S$ be a morphism of schemes. 
Then $\cC = \Dperf(X)$ has the structure of an $S$-linear category, where the 
action functor $\cC \times \Dperf(S) \to \Dperf(X)$ is given by $(E, F) \mapsto E \otimes f^*F$. If $T \to S$ is a morphism of schemes, then by \cite[Theorem 1.2]{bzfn} there is a $T$-linear equivalence 
\begin{equation*}
\cC_{T} \simeq \Dperf(X_{T}) 
\end{equation*} 
where $X_{T} = X \times_{S} T \to T$ denotes the derived fiber product, which agrees with the usual fiber product of schemes if $X \to S$ and $T \to S$ are $\Tor$-independent over $S$. 
\end{example} 

\subsection{Semiorthogonal decompositions} 
The above example can be amplified using the following observation. 
If $\cC \in \Cat_S$, a semiorthogonal decomposition 
\begin{equation} 
\label{C-sod}
\cC = \llangle \cC_1, \dots, \cC_m \rrangle 
\end{equation} 
is called $S$-linear if the $\Dperf(S)$-action preserves each of the components $\cC_i$. 
In this case, the $\cC_i$ inherit the structure of $S$-linear categories. 
In particular, if $X$ is an $S$-scheme, then $S$-linear semiorthogonal components of 
$\Dperf(X)$ are $S$-linear categories. 
This will be our main source of examples in the paper. 

By \cite[Lemma 3.15]{NCHPD}, given an $S$-linear semiorthogonal decomposition \eqref{C-sod} and a morphism $T \to S$, there is an induced $T$-linear semiorthogonal 
decomposition 
\begin{equation*}
\cC_T = \llangle (\cC_1)_T, \dots, (\cC_m)_T \rrangle.
\end{equation*} 
If $\cC = \Dperf(X)$ and $X$ and $T$ are $\Tor$-independent over $S$, the 
base changes $(\cC_i)_T$ can be expressed without the use of higher categories and derived algebraic geometry, by working inside of the ambient category 
$\Dperf(X_T)$, see \cite{kuznetsov-base-change}. 

The property that an $S$-linear subcategory $\cA \subset \cC$ forms part of a semiorthogonal decomposition can be characterized in terms of the embedding functor $\alpha \colon \cA \to \cC$. 
Namely, we say $\cA \subset \cC$ is a left admissible if $\alpha$ admits a left adjoint, 
right admissible if $\alpha$ admits a right adjoint, and admissible if $\alpha$ admits both adjoints. 
Then if $\cA, \cB \subset \cC$ are $S$-linear subcategories, we have a semiorthogonal 
decomposition $\cC = \llangle \cA, \cB \rrangle$ if and only if $\cA$ is left admissible and $\cB = {^\perp}\cA$, 
if and only if $\cB$ is right admissible and $\cA = \cB^{\perp}$. 

\subsection{Presentable linear categories} 
For technical reasons, it is sometimes useful to work with ``large'' versions of 
linear categories, which we review here; 
for clarity we sometimes say ``small $S$-linear category'' to mean an $S$-linear category in the sense of \S\ref{linear-categories}. 
Large categories will only be needed for our discussion of 
Hochschild (co)homology in \S\ref{section-HH-homology} and \S\ref{section-HH-cohomology}. 

A \emph{presentable $S$-linear category} $\cC$ is a presentable stable $\infty$-category 
$\cC$ equipped with a module structure over $\Dqc(S)$. 
As in the case of small linear categories, the collection of all such categories is organized into a 
symmetric monoidal $\infty$-category $\PrCat_S$, whose tensor product is denoted by 
\begin{equation*}
\cC \otimes_{\Dqc(S)} \cD \in \PrCat_S. 
\end{equation*} 
A morphism $\cC \to \cD$ in $\PrCat_S$ is a cocontinuous $S$-linear functor; these morphisms 
form the objects of a presentable $S$-linear category $\Fun_{S}(\cC, \cD)$, 
which is the internal mapping object in the category $\PrCat_S$. 

Many presentable $S$-linear categories which arise in practice are compactly generated, e.g. 
$\Dqc(S)$ is so by our assumption that $S$ is quasi-compact and quasi-separated \cite[Theorem~3.1.1]{bondal-vdb}. 
We denote by $\PrCat_{S}^{\omega}$ the $\infty$-category of compactly generated 
presentable $S$-linear categories, with morphisms the cocontinuous $S$-linear functors which 
preserve compact objects. Again, $\PrCat_{S}^{\omega}$ admits a symmetric monoidal 
structure and an internal mapping object 
$\Fun_{S}^{\omega}(\cC, \cD)$ for $\cC, \cD \in \PrCat_S^{\omega}$. 

The various versions of linear categories $\Cat_S, \PrCat_S,$ and $\PrCat_S^{\omega}$ are 
related as follows. By definition, $\PrCat_S^{\omega}$ is a non-full subcategory of $\PrCat_S$. 
Moreover, for any $\cC \in \Cat_S$ there is a category $\Ind(\cC) \in \PrCat_S^{\omega}$ 
called its Ind-completion, which roughly is 
obtained from $\cC$ by freely adjoining all filtered colimits. 
This gives a functor 
\begin{equation*}
\Ind \colon \Cat_S \to \PrCat_{S}^{\omega}
\end{equation*} 
which is in fact a symmetric monoidal equivalence with inverse the functor 
\begin{equation*} 
(-)^c \colon \PrCat_S^{\omega} \to \Cat_S
\end{equation*} 
taking $\cC \in \PrCat_S$ to its subcategory $\cC^c$ of compact objects. 

\begin{example}
\label{example-Dqc-X}
Let $f \colon X \to S$ be a morphism of schemes. 
Similar to Example~\ref{example-Dperf-X}, $\Dqc(X)$ naturally has the structure of a 
presentable $S$-linear category. 
In fact, if $X$ is quasi-compact and quasi-separated, then 
there is an equivalence 
$\Ind(\Dperf(X)) \simeq \Dqc(X)$ of presentable $S$-linear categories. 
\end{example}

\subsection{Mapping objects} 
For objects $E,F \in \cC$ of an $\infty$-category, we write $\Map_{\cC}(E,F)$ for the space of 
maps from $E$ to $F$. 
If $\cC$ is a presentable $S$-linear category, then 
there is a mapping object 
\begin{equation*}
\cHom_S(E,F) \in \Dqc(S) 
\end{equation*} 
characterized by equivalences 
\begin{equation*}
\Map_{\Dqc(S)}(G, \cHom_S(E,F)) \simeq \Map_{\cC}( E \otimes G, F ) 
\end{equation*} 
for $G \in \Dqc(S)$. 

If instead $\cC$ is an $S$-linear category, 
we write $\cHom_S(E, F) \in \Dqc(S)$ for the mapping object between $E$ and $F$ 
regarded as objects of the presentable $S$-linear category $\Ind(\cC)$; 
equivalently, $\cHom_S(E, F)$ can be characterized by equivalences 
\begin{equation*}
\Map_{\Dqc(S)}(G, \cHom_S(E,F)) \simeq \Map_{\cC}( E \otimes G, F) 
\end{equation*} 
for $G \in \Dperf(S)$. 
For $i \in \bZ$ we write 
$\cExt^i_S(E,F)$ for the degree $i$ cohomology sheaf of $\cHom_S(E,F)$, 
and 
$\Ext^i_S(E,F)$ for the degree $i$ hypercohomology of $\cHom_S(E,F)$. 

\begin{example}
Let $f \colon X \to S$ be a morphism of schemes. 
Then for $E,F \in \Dqc(X)$, we have 
\begin{equation*}
\cHom_S(E,F) \simeq f_* \cHom_X(E,F)
\end{equation*}
where $\cHom_X(E,F) \in \Dqc(X)$ denotes the derived sheaf Hom on $X$. 
\end{example} 

\subsection{Dualizable categories} 
\label{section-dualizable} 
Let $(\cA, \otimes, \one)$ be a symmetric monoidal $\infty$-category.  
An object $A \in \cA$ is called \emph{dualizable} if there exists an object $A^{\svee} \in \cA$ and 
morphisms 
\begin{equation*}
\coev_A \colon \one \to A \otimes A^{\svee} \qquad \text{and} \qquad \ev_A \colon A^{\svee} \otimes A \to \one 
\end{equation*} 
such that the compositions 
\begin{align*}
& A \xrightarrow{ \,  \coev_A \otimes \id_{A} \,} A \otimes A^{\svee} \otimes A \xrightarrow{\, \id_{A} \otimes \ev_A \,} A  , \\ 
& A^{\svee} \xrightarrow{ \, \id_{A^{\svee}} \otimes \coev_A \,} A^{\svee} \otimes A \otimes A^{\svee} \xrightarrow{\, \ev_A \otimes \id_{A^{\svee}} \,} A^{\svee} , 
\end{align*}
are equivalent to the identity morphisms of $A$ and $A^{\svee}$. 

\begin{remark}
\label{remark-duality-data} 
Dualizability of an object $A \in \cA$ is detected at the level of the homotopy category $\rh \cA$; 
moreover, if $A$ is dualizable, then the object $A^{\svee}$ and the evalation and coevaluation morphisms 
are uniquely determined in $\rh\cA$. 
\end{remark}

The following gives a large source of dualizable presentable linear categories. 

\begin{lemma}[{\cite[Lemma 4.3]{NCHPD}}]
\label{lemma-presentable-dualizable} 
Let $\cC$ be a compactly generated presentable $S$-linear category. 
Then $\cC$ is dualizable as an object of $\PrCat_S$, with dual given by 
\begin{equation*}
\cC^{\svee} = \Ind((\cC^c)^{\op}) 
\end{equation*} 
where $(\cC^c)^{\op}$ denotes the opposite of the category $\cC^{c}$ of compact objects in $\cC$.  
There is a canonical equivalence $\cC \otimes_{\Dqc(S)} \cC^{\svee} \simeq \Fun_S(\cC, \cC)$ under which the coevaluation morphism 
\begin{equation*}
\coev_{\cC} \colon \Dqc(S) \to \cC \otimes_{\Dqc(S)} \cC^{\svee} 
\end{equation*} 
is the canonical functor sending $\cO_S \in \Dqc(S)$ to $\id_{\cC} \in \Fun_S(\cC, \cC)$. 
The evaluation morphism 
\begin{equation*}
\ev_{\cC} \colon \cC^{\svee} \otimes_{\Dqc(S)} \cC \to \Dqc(S) 
\end{equation*} 
is induced by the functor $\cHom_{S}(-,-) \colon (\cC^c)^{\op} \times \cC^c \to \Dqc(S)$. 
\end{lemma}

In particular, the lemma implies the following, where recall that by convention all schemes are quasi-compact and quasi-separated. 
\begin{corollary}
If $f \colon X \to S$ is a morphism of schemes, then 
$\Dqc(X)$ is a dualizable presentable $S$-linear category. 
\end{corollary} 

Dualizability of a small $S$-linear category is more restrictive. 
Recall that if $\cC$ is a small $S$-linear category, then $\cC$ is called: 
\begin{itemize}
\item \emph{proper (over $S$)} if 
$\cHom_S(E,F) \in \Dperf(S) \subset \Dqc(S)$ for all $E,F \in \cC$, and 
\item \emph{smooth (over $S$)} if $\id_{\Ind(\cC)} \in \Fun_{\Dqc(S)}(\Ind(\cC), \Ind(\cC))$ is a compact object. 
\end{itemize} 
Moreover, $\cC$ is dualizable as an object of $\Cat_S$ if and only if $\cC$ is smooth and proper over $S$, in which case the dual is given by $\cC^{\svee} = \cC^{\op}$ \cite[Lemma 4.8]{NCHPD}. 

This is closely related to the usual notions of smoothness and properness in geometry. 
For instance, if $f \colon X \to S$ is a smooth and proper morphism, 
then $\Dperf(X)$ is smooth and proper over $S$ \cite[Lemma 4.9]{NCHPD}. 
Further, semiorthogonal components of a smooth proper $S$-linear category are automatically smooth, proper, and admissible 
 \cite[Lemma 4.15]{NCHPD}. 
Putting these observations together gives the following key examples of smooth and proper linear categories for this paper.  

\begin{lemma}
Let $f \colon X \to S$ be a smooth proper morphism.  
If $\cC$ is an $S$-linear semiorthogonal component of $\Dperf(X)$, then 
$\cC$ is smooth and proper over $S$, and the embedding $\cC \hookrightarrow \Dperf(X)$ 
is admissible. 
\end{lemma} 

Smooth and proper categories enjoy many nice properties. 
For instance, a smooth proper $S$-linear category $\cC$ always admits a 
\emph{Serre functor $\rS_{\cC/S}$ over $S$} \cite[Lemma 4.19]{NCHPD}. 
By definition, this means $\rS_{\cC/S}$ is an autoequivalence of $\cC$ 
such that there are natural equivalences 
\begin{equation*}
\cHom_S(E, \rS_{\cC/S}(F)) \simeq \cHom_S(F,E)^{\svee} 
\end{equation*} 
for $E, F \in \cC$. 
For example, if $f \colon X \to S$ is a smooth proper morphism 
of relative dimension~$n$, 
then $\rS_{\Dperf(X)/S} = - \otimes \omega_{X/S}[n]$ is a Serre functor over $S$. 


\section{Hochschild homology} 
\label{section-HH-homology}

In this section we review the definition of Hochschild homology and various of its properties 
relevant to this paper. 
All of the constructions and results we discuss are well-known in some form, but 
for convenience or lack of suitable references we often sketch the details. 

There are various settings in which Hochschild homology can be defined. 
In this paper, we consider Hochschild homology as an invariant of small linear categories or dualizable presentable linear categories, defined in terms of categorical traces. 
See \cite{kuznetsov-HH-sod} for a more down-to-earth definition in the case of a semiorthogonal component of the derived category of a smooth proper variety, 
which is the case needed in the main results of this paper.  
The definition below has the advantage of being manifestly canonical and convenient for making abstract arguments. 
 
In general, if $(\cA, \otimes, \one)$ is a symmetric monoidal $\infty$-category 
and $A \in \cA$ is a dualizable object, then the \emph{trace} of an endomorphism $F \colon A \to A$ is the 
map $\Tr(F) \in \Map_{\cA}(\one, \one)$ given as the 
composite 
\begin{equation*}
\one \xrightarrow{\, \coev_A \,} A \otimes A^{\svee} \xrightarrow{\, F \otimes \id_{A^{\svee}} \,} A \otimes A^{\svee} \simeq A^{\svee} \otimes A \xrightarrow{\, \ev_{A} \,} \one. 
\end{equation*} 
We will be interested in the case where $\cA$ is $\Cat_S$ or $\PrCat_S$. 
In this case, $\one$ is $\Dperf(S)$ or $\Dqc(S)$, and the functor  
$\Tr(F)$ is determined by its value on the structure sheaf $\cO_S$. 

\begin{definition}
Let $\cC$ be a dualizable 
presentable $S$-linear category, and let $F \colon \cC \to \cC$ be an endomorphism. 
Then the \emph{Hochschild homology of $\cC$ over $S$ with coefficients in $F$} is 
the complex 
\begin{equation*}
\HH_*(\cC/S, F) = \Tr(F)(\cO_S) \in \Dqc(S). 
\end{equation*} 
The \emph{Hochschild homology of $\cC$ over $S$} is the complex 
\begin{equation*}
\HH_*(\cC/S) = \HH_*(\cC/S, \id_{\cC}) \in \Dqc(S). 
\end{equation*} 
If $\cC$ is an $S$-linear category and $F \colon \cC \to \cC$ is an endomorphism, 
then $\Ind(\cC)$ is a dualizable presentable $S$-linear category by Lemma~\ref{lemma-presentable-dualizable},  and we define 
\begin{align*}
\HH_*(\cC/S, F) & = \HH_*(\Ind(\cC)/S, \Ind(F)) , \\ 
\HH_*(\cC/S) & = \HH_*(\cC/S, \id_{\cC}) . 
\end{align*} 
Note that any object $F \in \Dqc(S)$ gives a natural coefficient for Hochschild homology of categories over $S$, by considering the corresponding endofunctor $- \otimes F \colon \cC \to \cC$; in this situation, we use the notation  
\begin{equation*}
\HH_*(\cC/S, F) = \HH_*(\cC/S, (- \otimes F)). 
\end{equation*}
Finally, in any of the above situations, for $i \in \bZ$ we set  
\begin{equation*}
\HH_i(\cC/S, F) = \rH^{-i}(\HH_*(\cC/S, F))
\end{equation*} 
to be the degree $-i$ cohomology sheaf of $\HH_*(\cC/S,F)$. 
\end{definition}

\begin{remark}
\label{remark-HH-dualizable}
If $\cC$ is a dualizable small $S$-linear category (equivalently, a smooth and proper small $S$-linear category, see \S\ref{section-dualizable}) and $F \colon \cC \to \cC$ is an endomorphism, then by definition the trace of $F$ is a functor $\Tr(F) \colon \Dperf(S) \to \Dperf(S)$. 
Further, there is a canonical equivalence $\Ind(\Tr(F)) \simeq \Tr(\Ind(F))$, 
because by Remark~\ref{remark-duality-data} the functor $\Ind$ takes the duality data of $\cC$ to that of $\Ind(\cC)$.  
Thus $\HH_*(\cC/S, F) \simeq \Tr(F)(\cO_S) \in \Dperf(S)$. 
\end{remark}

Below we review some well-known properties of Hochschild homology, in a guise that is tailored to our purposes. 

\subsection{Functoriality} 
\label{HH_*-functorial} 
Hochschild homology (with coefficients) is suitably functorial. 
This functoriality exists in the general context of traces of dualizable objects in a symmetric monoidal $(\infty,2)$-category, 
see e.g \cite{nonlinear-traces, toen-traces, higher-traces, categorical-atiyah-bott}, but here we 
only recall the relevant details in the case 
of Hochschild homology of categories. 

Namely, let $(\cC, F)$ be a pair where $\cC$ is a dualizable presentable $S$-linear category and 
$F \colon \cC \to \cC$ is an endomorphism. 
Let $(\cD, G)$ by another such pair. 
We define a morphism $(\cC, F) \to (\cD, G)$ to be a pair $(\Phi, \gamma)$ where 
$\Phi \colon \cC \to \cD$ is morphism that admits a cocontinuous right adjoint $\Phi^!$ (which is thus also a morphism in $\PrCat_S$), 
and $\gamma \colon \Phi \circ F \to G \circ \Phi$ is a natural transformation of functors; in other words, 
a morphism is a (not necessarily commutative) diagram 
\begin{equation*}
\xymatrix{
\cC \ar[r]^{F} \ar[d]_{\Phi} & \cC \ar[d]^{\Phi}  \ar@{=>}[dl]_{\gamma} \\ 
\cD \ar[r]_{G} & \cD 
}  
\end{equation*}
Given such a morphism $(\Phi, \gamma)$, we consider the diagram
\begin{equation}
\label{tr-functorial} 
\vcenter{
\xymatrix{
\Dqc(S) \ar[rr]^{\coev_{\cC}} \ar@{=}[dd] & & \cC \otimes_{\Dqc(S)} \cC^{\svee} \ar@{=>}[ddll] \ar[rr]^{F \otimes \id_{\cC^{\svee}}} \ar[dd]_{\Phi \otimes (\Phi^{!})^{\svee}} & & \cC \otimes_{\Dqc(S)} \cC^{\svee} \ar[dd]^{\Phi \otimes (\Phi^{!})^{\svee}} \ar@{=>}[ddll]_{\gamma \otimes \id_{(\Phi^{!})^{\svee}}} \ar[rr]^{\ev_{\cC}} && \Dqc(S) \ar@{=}[dd] \ar@{=>}[ddll] \\ \\
\Dqc(S) \ar[rr]_{\coev_{\cD}} & & \cD \otimes_{\Dqc(S)} \cD^{\svee} \ar[rr]_{G \otimes \id_{\cD^{\svee}}} & &  \cD \otimes_{\Dqc(S)} \cD^{\svee} \ar[rr]_{\ev_{\cD}}&& \Dqc(S)
}
}
\end{equation} 
where: 
\begin{itemize}
\item $(\Phi^!)^{\svee} \colon \cC^{\svee} \to \cD^{\svee}$ is the dual of the functor $\Phi^! \colon \cD \to \cC$, defined as the composition  
\begin{equation*}
\cC^{\svee} \xrightarrow{ \id_{\cC^{\svee}} \otimes \coev_{\cD}} 
\cC^{\svee} \otimes_{\Dqc(S)} \cD \otimes_{\Dqc(S)} \cD^{\svee} 
\xrightarrow{ \id_{\cC^{\svee}} \otimes \Phi^! \otimes \id_{\cD^{\svee}}} 
\cC^{\svee} \otimes_{\Dqc(S)} \cC \otimes_{\Dqc(S)} \cD^{\svee} 
\xrightarrow{\ev_{\cC} \otimes \id_{\cD^{\svee}}} \cD^{\svee}  . 
\end{equation*} 
\item The $2$-morphism in the first square is the natural transformation  
\begin{equation*}
(\Phi \otimes (\Phi^{!})^{\svee}) \circ \coev_{\cC} \simeq 
((\Phi \circ \Phi^!) \otimes \id_{\cD^{\svee}}) \circ \coev_{\cD} 
\to \coev_{\cD}
\end{equation*} 
induced by the counit of the adjunction between $\Phi$ and $\Phi^!$. 
\item The $2$-morphism in the last square is the natural transformation 
\begin{equation*} 
\ev_{\cC} \to \ev_{\cC} \circ ((\Phi^! \circ \Phi) \otimes \id_{\cC^{\svee}}) \simeq \ev_{\cD} \circ (\Phi \otimes (\Phi^!)^{\svee}) 
\end{equation*} 
induced by the unit of the adjunction between $\Phi$ and $\Phi^!$. 
\end{itemize}
The compositions along the top and bottom of~\eqref{tr-functorial} are 
by definition the traces $\Tr(F)$ and $\Tr(G)$, so the composition of the $2$-morphisms in the diagram gives a natural transformation 
\begin{equation*}
\Tr(\Phi, \gamma) \colon \Tr(F) \to \Tr(G). 
\end{equation*} 
In particular, applying this to $\cO_S$, we obtain a morphism on 
Hochschild homology 
\begin{equation*}
\HH_*(\Phi, \gamma) \colon \HH_*(\cC/S, F) \to \HH_*(\cD/S, G). 
\end{equation*} 

The functoriality of Hochschild homology implies the following 
result, cf. \cite{kuznetsov-HH-sod} which treats the case of semiorthogonal decompositions of varieties. 

\begin{lemma}
Let $\cC = \llangle \cC_1, \dots, \cC_m \rrangle$ be an $S$-linear semiorthogonal decomposition with admissible components. 
Then there is an equivalence  
\begin{equation*}
\HH_*(\cC/S) \simeq \HH_*(\cC_1/S) \oplus \cdots \oplus \HH_*(\cC_m/S), 
\end{equation*} 
where the map $\HH_*(\cC/S) \to \HH_*(\cC_i/S)$ is induced by the projection functor onto the 
component $\cC_i$. 
\end{lemma} 

\subsection{Chern characters} 
\label{section-chern-characters}
The functoriality of Hochschild homology can be used to define a theory of Chern characters, as follows. 

Let $\cC$ be a presentable $S$-linear category. 
Then any object $E \in \cC$ determines an $S$-linear functor $\Phi_{E} \colon \Dqc(S) \to \cC$ determined by $\Phi_{E}(\cO_S) = E$, whose right adjoint 
\begin{equation*}
\Phi_{E}^! = \cHom_S(E, -) \colon \cC \to \Dqc(S)
\end{equation*} 
is cocontinuous if and only if $E$ is a compact object of $\cC$. 

Now we assume that $E$ is compact, that $\cC$ is dualizable, and $F \colon \cC \to \cC$ is an endomorphism; for instance, $\cC$ could be of the form $\cC = \Ind(\cC_0)$ for a small $S$-linear 
category $\cC_0$ and $E \in \cC_0$. 
In this setup, we will construct a morphism 
\begin{equation*}
\ch_{E,F} 
\colon \cHom_S(E, F(E)) \to \HH_*(\cC/S, F)
\end{equation*} 
in $\Dqc(S)$, called the \emph{Chern character of $E$ with coefficients in $F$}; 
in practice, we often drop the subscripts $E$ and $F$ in $\ch_{E,F}$ when they are clear from context. 
By Yoneda's Lemma, it suffices to construct functorially in $G \in \Dqc(S)$ a map 
\begin{equation}
\label{ch-yoneda}
\Map_{\Dqc(S)}(G, \cHom_S(E, F(E))) \to \Map_{\Dqc(S)}(G, \HH_*(\cC/S, F)). 
\end{equation} 
The left side is identified with $\Map_{\cC}(E \otimes G, F(E))$. 
This mapping space is in turn identified with the space of natural transformations 
$\gamma \colon \Phi_{E} \circ (- \otimes G) \to F \circ \Phi_E$. 
The pair $(\Phi_E, \gamma)$ is then a morphism of pairs $(\Dqc(S), - \otimes G) \to (\cC, F)$ as 
considered in \S\ref{HH_*-functorial}, and hence determines a morphism on Hochschild homology 
\begin{equation*}
G \simeq \HH_*(\Dqc(S)/S, G) \to \HH_*(\cC/S, F).  
\end{equation*} 
All together, this gives the required map~\eqref{ch-yoneda}.  

\subsection{Base change} 
Hochschild homology satisfies base change in the following sense. 
\begin{lemma}
\label{lemma-HH-bc} 
Let $\cC$ be a dualizable presentable $S$-linear category and 
let $F \colon \cC \to \cC$ be an endomorphism. 
Let $g \colon T \to S$ be a morphism of schemes. 
Let $F_T \colon \cC_T \to \cC_T$ be the base change of $F$ along $g$. 
Then there is a canonical equivalence 
\begin{equation*} 
g^*\HH_*(\cC/S, F) \simeq \HH_*(\cC_T/T, F_T) .
\end{equation*} 
\end{lemma}

\begin{proof}
It follows from the definitions that the trace $\Tr(F) \colon \Dqc(S) \to \Dqc(S)$ commutes with 
base change, which implies the result. 
\end{proof} 

For smooth proper categories in characteristic $0$, the individual Hochschild homology groups are vector bundles and satisfy base change. 

\begin{theorem}[\cite{kaledin1, kaledin2, akhil-degeneration}]
\label{theorem-degeneration} 
Let $\cC$ be a smooth proper $S$-linear category, where $S$ is a $\bQ$-scheme. 
Then $\HH_i(\cC/S)$ is a finite locally free sheaf on $S$ for any $i \in \bZ$. 
Further, if $g \colon T \to S$ is a morphism of schemes, then for any $i \in \bZ$ there is a canonical isomorphism 
\begin{equation*}
g^*\HH_i(\cC/S) \cong \HH_i(\cC_T/T). 
\end{equation*} 
\end{theorem} 

\begin{proof}
The first part follows from the degeneration of the noncommutative Hodge-to-de Rham spectral sequence proved by Kaledin \cite{kaledin1, kaledin2}; see also Mathew's recent proof \cite[Theorem 1.3]{akhil-degeneration}. 
The second claim then follows from Lemma~\ref{lemma-HH-bc}. 
\end{proof}

\subsection{Mukai pairing} 
In the smooth and proper case, Hochschild homology carries a canonical 
nondegenerate pairing, known as the Mukai pairing. 
This pairing has been studied from many points of view in the literature \cite{caldararu1, caldararu2, shklyarov, markarian}. 

\begin{lemma}
\label{lemma-mukai} 
Let $\cC$ be a smooth proper $S$-linear category. 
Then $\HH_*(\cC/S) \in \Dperf(S)$ and 
there is a canonical nondegenerate pairing $\HH_*(\cC/S) \otimes \HH_*(\cC/S) \to \cO_S$. 
\end{lemma}

\begin{proof}
We sketch a short proof following \cite{antieau-vezzosi}. 
(We already observed $\HH_*(\cC/S) \in \Dperf(S)$ in 
Remark~\ref{remark-HH-dualizable}, but the following gives another proof.) 
The functor $\HH_*(-/S) \colon \Cat_S \to \Dqc(S)$ is symmetric monoidal; 
see for instance \cite[Proposition 2.1]{antieau-vezzosi} (where the result is stated for $S$ affine, 
but from which the general case follows by Lemma~\ref{lemma-HH-bc}). 
If $\cC$ is smooth and proper over $S$, then it is dualizable as an object of $\Cat_S$. 
Since $\HH_*(-/S)$ is symmetric monoidal, its value on $\cC$ is also dualizable as an object of $\Dqc(S)$, 
and hence belongs to $\Dperf(S)$. 
The evaluation morphism for $\HH_*(\cC/S)$ is obtained by applying the functor $\HH_*(-/S)$ to the 
evaluation morphism for $\cC$, and hence takes the form 
\begin{equation*} 
\HH_*(\cC^{\svee}/S) \otimes \HH_*(\cC/S) \to \cO_S . 
\end{equation*} 
But it follows from the definition of Hochschild homology that there is a canonical identification 
$\HH_*(\cC^{\svee}/S) \simeq \HH_*(\cC/S)$. This completes the proof. 
\end{proof} 

We will need a compatibility between the Mukai pairing, Serre duality, and Chern characters, which 
we formulate in the case where the base is a field. 

\begin{lemma}
\label{lemma-mukai-serre}
Let $\cC$ be a smooth proper $k$-linear category, where $k$ is a field. 
Then for $i \in \bZ$ there is an isomorphism 
\begin{equation*}
\HH_i(\cC/k) \cong \Ext^{-i}_{k}(\id_{\cC}, \rS_{\cC}) 
\end{equation*} 
where $\rS_{\cC}$ is the Serre functor for $\cC$ over $k$ and the $\Ext$ group is considered 
in the category of $k$-linear endofunctors of $\cC$. 
Moreover, for $E \in \cC$ if we denote by 
\begin{equation*}
\eta_E \colon \HH_i(\cC/k) \to \Ext^{-i}_k(E, \rS_{\cC}(E)) 
\end{equation*} 
the natural map arising from the above isomorphism, 
then there is a commutative diagram 
\begin{equation*}
\xymatrix{
\Ext^i_k(E, E) \ar[r]^{\ch_E} \ar[d]_{\cong} & \HH_{-i}(\cC/k) \ar[d]^{\cong} \\ 
\Ext^{-i}_k(E, \rS_{\cC}(E))^{\svee} \ar[r]^{ \, \eta_E^{\svee}} & \HH_i(\cC/k)^{\svee} 
}
\end{equation*} 
where the left vertical arrow is given by Serre duality and the right vertical arrow is give by the Mukai pairing. 
\end{lemma}

\begin{proof}
This is well-known to the experts. 
We provide references to the literature where the statements are proved in a slightly different setup (e.g. $\cC$ assumed to be a semiorthogonal 
component in the derived category of a variety): 
the isomorphism $\HH_i(\cC/k) \cong \Ext^{-i}_{k}(\id_{\cC}, \rS_{\cC})$ follows 
from \cite[Theorem 4.5 and Proposition 4.6]{kuznetsov-HH-sod}, 
and the commutativity of the diagram follows from \cite[Proposition 11]{caldararu1}. 
\end{proof} 

\subsection{HKR isomorphism} 
The Hochschild--Kostant--Rosenberg (HKR) isomorphism identifies the Hochschild homology of the derived category of a scheme in terms of Hodge cohomology. 
This subject has been studied by many authors, see e.g. \cite{HKR, swan, yekutieli}; 
the form in which we state the result is a 
consequence of Yekutieli's work \cite{yekutieli}. 
For a morphism $X \to S$ and an endomorphism $F \colon \Dperf(X) \to \Dperf(X)$, we 
use the notation $\HH_*(X/S, F) = \HH_*(\Dperf(X)/S, F)$. 

\begin{theorem}[\cite{yekutieli}]
Let $f \colon X \to S$ be smooth morphism of relative dimension $n$, where $n!$ is invertible on $S$. 
Let $F \in \Dperf(S)$. 
Then there is an equivalence 
\begin{equation*}
\HH_*(X/S, F) \simeq \bigoplus_{p = 0}^n F \otimes f_* \Omega^{p}_{X/S}[p] . 
\end{equation*}
\end{theorem}


\section{Hochschild cohomology}
\label{section-HH-cohomology}

In this section we review the definition of Hochschild cohomology and various of its properties 
relevant to this paper.  
As in our discussion of Hochschild homology, this material is well-known but for convenience 
we often sketch the details. 

\begin{definition}
Let $\cC$ be a small or presentable $S$-linear category, 
and let $F \colon \cC \to \cC$ be an endomorphism. 
Then the \emph{Hochschild cohomology of $\cC$ over $S$ with coefficients in $F$} is the complex 
\begin{equation*}
\HH^*(\cC/S, F) = \cHom_S(\id_{\cC}, F) \in \Dqc(S), 
\end{equation*} 
i.e. the mapping object from $\id_{\cC}$ to $F$ considered as objects of the 
$S$-linear category $\Fun_{S}(\cC, \cC)$. 
The \emph{Hochschild cohomology of $\cC$ over $S$} is the complex 
\begin{equation*}
\HH^*(\cC/S) = \HH^*(\cC/S, \id_{\cC}) \in \Dqc(S). 
\end{equation*} 
As for Hochschild homology, if $F \in \Dqc(S)$ we use the notation 
\begin{equation*}
\HH^*(\cC/S, F) = \HH^*(\cC/S, (- \otimes F)). 
\end{equation*} 
Finally, for $i \in \bZ$ we set  
\begin{equation*}
\HH^i(\cC/S, F) = \rH^{i}(\HH_*(\cC/S, F))
\end{equation*} 
to be the degree $i$ cohomology sheaf of $\HH^*(\cC/S,F)$. 
\end{definition} 

\begin{remark}
If $\cC$ is a small $S$-linear category and $F \colon \cC \to \cC$ is an endofunctor, 
then the Hochschild cohomologies $\HH^*(\cC, F)$ and $\HH^*(\Ind(\cC), \Ind(F))$ are canonically equivalent; 
this follows from the fact that $\Ind \colon \Cat_S \to \PrCat^{\omega}_S$ is an equivalence. 
\end{remark} 

\subsection{Functoriality} 
\label{HH^*-functorial}
As recalled in \S\ref{HH_*-functorial}, Hochschild homology is functorial with respect to functors that admit a right adjoint. 
Hochschild cohomology, however, is only functorial with respect to functors which are also fully faithful. 

Namely, let $(\cC, F)$ be a small or presentable $S$-linear category, and let $F \colon \cC \to \cC$ be an $S$-linear endofunctor. 
Let $(\cD, G)$ be another such pair. 
We consider pairs 
$(\Phi, \delta)$ where $\Phi \colon \cC \to \cD$ is a morphism that admits a right adjoint morphism $\Phi^! \colon \cD \to \cD$ (so $\Phi^!$ is required to be cocontinuous in case $\cC$ and $\cD$ are presentable), and $\delta \colon G \circ \Phi \to \Phi \circ F$ is a natural transformation of functors; 
in other words, we consider a (not necessarily commutative) diagram 
\begin{equation*}
\xymatrix{
\cC \ar[r]^{F} \ar[d]_{\Phi} & \cC \ar[d]^{\Phi}   \\ 
\cD \ar[r]_{G} \ar@{=>}[ur]^{\delta} & \cD 
}   
\end{equation*}
To distinguish from the notion of a morphism $(\cC, F) \to (\cD, G)$ introduced in \S\ref{HH_*-functorial}, we will call such a pair $(\Phi, \delta)$ a \emph{co-morphism} from $(\cC, F)$ to $(\cD, G)$. We say that $(\Phi, \delta)$ is fully faithful if $\Phi$ is fully faithful. 

Now assume that $(\Phi, \delta) \colon (\cC, F) \to (\cD, G)$ is a fully faithful co-morphism. 
In this setup, we will construct a morphism 
\begin{equation*}
\HH^*(\Phi, \delta) \colon \HH^*(\cD, G) \to \HH^*(\cC, F) 
\end{equation*} 
in $\Dqc(S)$. By the definition of Hochschild cohomology and Yoneda, 
it suffices to construct functorially in $A \in \Dqc(S)$ a map 
\begin{equation*}
\Map_{\Fun_{S}(\cD, \cD)}((- \otimes A), G) \to
\Map_{\Fun_{S}(\cC, \cC)}((- \otimes A), F) . 
\end{equation*} 
For this, we send $\alpha \colon (- \otimes A) \to G$ on the left side to the 
morphism $(- \otimes A) \to F$ given by the composition 
\begin{equation*}
(- \otimes A) \xrightarrow{ \sim } 
\Phi^! \circ \Phi \circ (- \otimes A) \xrightarrow{\sim} 
\Phi^! \circ (- \otimes A) \circ \Phi \xrightarrow{\Phi^! \alpha \Phi} 
\Phi^! \circ G \circ \Phi  \xrightarrow{\Phi^! \delta}
\Phi^! \circ \Phi \circ F \xrightarrow{\sim}  
F , 
\end{equation*} 
where the first and last equivalences come from fully faithfulness of $\Phi$ and 
the second equivalence from the $S$-linearity of $\Phi$. 

\subsection{Base change} 
Like Hochschild homology, Hochschild cohomology satisfies base change. 
This will not be needed in the paper, but we include it for completeness. 

\begin{lemma}
\label{lemma-HcoH-bc} 
Let $\cC$ be a dualizable presentable $S$-linear category and 
let $F \colon \cC \to \cC$ be an endomorphism. 
Let $g \colon T \to S$ be a morphism of schemes. 
Let $F_T \colon \cC_T \to \cC_T$ be the base change of $F$ along $g$. 
Then there is a canonical equivalence 
\begin{equation*} 
g^*\HH^*(\cC/S, F) \simeq \HH^*(\cC_T/T, F_T) . 
\end{equation*} 
\end{lemma}

\begin{proof}
We claim that the natural functor 
\begin{equation*}
\Fun_{S}(\cC, \cC) \otimes_{\Dqc(S)} \Dqc(T) \to \Fun_T(\cC_T, \cC_T) 
\end{equation*}  
is an equivalence of $T$-linear categories. 
From this, the lemma follows from the definition of Hochschild cohomology and 
base change for mapping objects in linear categories (see \cite[Lemma 2.10]{NCHPD}). 

To prove the claim, note that by dualizability of $\cC$ we have an equivalence 
\begin{equation*}
\Fun_{S}(\cC, \cC) \simeq \cC^{\svee} \otimes_{\Dqc(S)} \cC . 
\end{equation*} 
Base change along $T \to S$ preserves dualizability of $\cC$ and $(\cC_T)^{\svee} \simeq (\cC^{\svee})_T$, so we similarly have an equivalence 
\begin{equation*}
\Fun_{T}(\cC_T, \cC_T) \simeq (\cC^{\svee})_T \otimes_{\Dqc(T)} \cC_T . 
\end{equation*} 
Since these descriptions of the functor categories are compatible with base change along $T$, 
the claim follows. 
\end{proof}

\subsection{Action on homology} 
Hochschild cohomology acts on Hochschild homology. 
More generally, suppose $\cC$ is a dualizable presentable $S$-linear category, and $F, G \colon \cC \to \cC$ are endomorphisms. 
Then there is an action functor 
\begin{equation*}
\cHom_S(F, G) \otimes \HH_*(\cC/S, F) \to \HH_*(\cC/S, G) 
\end{equation*}
where the first term is the mapping object from $F$ to $G$ in $\Fun_S(\cC,\cC)$. 
This boils down to assigning to any natural transformation $\gamma \colon F \to G$ a morphism $\HH_*(\cC/S, F) \to \HH_*(\cC/S, G)$; 
since $(\id_{\cC}, \gamma) \colon (\cC, F) \to (\cC, G)$ is a morphism of pairs in the sense of~\S\ref{HH_*-functorial}, we can simply take $\HH_*(\id_{\cC}, \gamma)$. 

Note that as a particular case, we have an action 
\begin{equation*}
\HH^*(\cC/S, F) \otimes \HH_*(\cC/S) \to \HH_*(\cC/S, F), 
\end{equation*} 
and hence for any $i, j \in \bZ$ an action 
\begin{equation*}
\HH^i(\cC/S, F) \otimes \HH_{j}(\cC/S) \to \HH_{j-i}(\cC/S, F). 
\end{equation*}
For any $E \in \cC$, there is also an evident action 
\begin{equation*}
\HH^*(\cC/S, F) \otimes \cHom_S(E,E) \to \cHom_S(E,F(E)). 
\end{equation*}

\begin{lemma}
\label{lemma-ch-functoriality}
For $E \in \cC$ the diagram 
\begin{equation*}
\xymatrix{
\HH^*(\cC/S, F) \otimes \cHom_S(E,E)   \ar[r] \ar[d]_{\id \otimes \ch_{E}} & \cHom_S(E,F(E)) \ar[d]^{\ch_{E,F}} \\ 
\HH^*(\cC/S, F) \otimes \HH_*(\cC/S) \ar[r] & \HH_*(\cC/S, F)
}
\end{equation*} 
commutes. 
\end{lemma}

\begin{proof}
Using the functoriality of traces (see \cite[Proposition 3.21]{nonlinear-traces} or \cite[Proposition 1.2.11]{categorical-atiyah-bott}), this follows by unwinding the definitions. 
\end{proof}

\subsection{HKR isomorphism} 
The HKR isomorphism for Hochschild cohomology identifies this group in the case of the derived category of a scheme with polyvector field cohomology. 
Like the HKR isomorphism for Hochschild homology, the following form of this result can be deduced from \cite{yekutieli}. 
For a morphism $X \to S$ and an endomorphism 
$F \colon \Dperf(X) \to \Dperf(X)$, we use the notation $\HH^*(X/S, F) = \HH^*(\Dperf(X)/S, F)$.

\begin{theorem}
\label{theorem-HKR-cohomology}
Let $f \colon X \to S$ be smooth morphism of relative dimension $n$, where $n!$ is invertible on $S$. 
Let $F \in \Dperf(S)$. 
Then there is an equivalence 
\begin{equation*}
\HH^*(X/S, F) \simeq \bigoplus_{t = 0}^n F \otimes f_*(\wedge^t \rT_{X/S})[-t] . 
\end{equation*}
\end{theorem}

\subsection{Deformation theory} 
Let $0 \to I \to A' \to A \to 0$ be a square-zero extension of rings, and let $X \to \Spec(A)$ be a smooth morphism of schemes. 
A deformation of $X$ over $A'$ is a smooth scheme $X'$ over $A'$ equipped with an isomorphism $X'_{A} \cong X$. 
Recall that, provided one exists, the set of isomorphism classes of such deformations 
form a torsor under $\rH^1(X, \rT_{X/\Spec(A)} \otimes I)$ (where we abusively write $I$ for the pullback of $I$ to $X$). 
If $A' \to A$ is a trivial square-zero extension, i.e. admits a section $A \to A'$, 
then there is a trivial deformation $X_{A'}$ obtained by base change along the section, 
so there is a canonical identification of the set of deformations of $X$ over $A'$ with $\rH^1(X, \rT_{X/\Spec(A)} \otimes I)$ taking the trivial deformation to $0$;
in this case, for a deformation $X' \to \Spec(A')$ we write 
\begin{equation*}
\kappa(X') \in \rH^1(X, \rT_{X/\Spec(A)} \otimes I)
\end{equation*} 
for the corresponding element, called the \emph{Kodaira-Spencer class}. 

We will need a generalization of the Kodaira--Spencer class to the setting of categories. 
Note that by Theorem~\ref{theorem-HKR-cohomology}, 
if $\dim(X/A)!$ is invertible on $A$ (where $\dim(X/A)$ denotes the relative dimension of $X \to \Spec(A)$), 
then we have an isomorphism 
\begin{equation*}
\HH^2(X/A, I) \cong \rH^0(X, \wedge^2\rT_{X/\Spec(A)} \otimes I) 
\oplus \rH^1(X, \rT_{X/\Spec(A)} \otimes I) 
\oplus \rH^2(X, I) , 
\end{equation*} 
and in particular a natural inclusion 
\begin{equation}
\label{H1-HH2}
\rH^1(X, \rT_{X/\Spec(A)} \otimes I) \hookrightarrow \HH^2(X/A, I). 
\end{equation} 
This suggests that when we replace $X$ by an $A$-linear category, the 
role of the cohomology of $\rT_{X/\Spec(A)}$ in deformation theory should be 
replaced by Hochschild cohomology. 

If $A' \to A$ is a square-zero extension and $\cC$ is an $A$-linear category, 
then a \emph{deformation of $\cC$ over $A'$} is an $A'$-linear category $\cC'$ equipped with 
an equivalence $\cC'_{A} \simeq \cC$. 
If $\Phi \colon \cC \to \cD$ is a morphism of $A$-linear categories, then a 
\emph{deformation of $\Phi$ over $A'$} is a morphism $\Phi' \colon \cC' \to \cD'$ where $\cC'$ and $\cD'$ are deformations of $\cC$ and $\cD$ over $A'$ and the base change $\Phi'_{A}$ is equipped with an equivalence $\Phi'_{A} \simeq \Phi$. 

\begin{lemma}
\label{lemma-KS} 
Let $0 \to I \to A' \to A \to 0$ be a trivial square-zero extension of rings, 
and let $\cC$ be an $A$-linear category. 
Then for any deformation $\cC'$ of $\cC$ over $A'$, there is an associated Kodaira--Spencer class 
\begin{equation*}
\kappa(\cC') \in \HH^2(\cC/A, I) 
\end{equation*} 
with the following properties: 
\begin{enumerate}
\item 
\label{kappa1}
Let $\Phi \colon \cC \to \cD$ be a fully faithful morphism of $A$-linear categories which admits a right adjoint. 
Let $\Phi' \colon \cC' \to \cD'$ be a deformation of $\Phi$ over $A'$. 
Then the map 
\begin{equation*}
\HH^2(\cD/A, I) \to \HH^2(\cC/A, I)
\end{equation*} 
induced by $\Phi$ (see \S\ref{HH^*-functorial}) takes $\kappa(\cD')$ to $\kappa(\cC')$. 

\item
\label{kappa2} 
Let $X \to \Spec(A)$ be a smooth morphism of schemes 
with $\dim(X/A)!$ invertible on $A$. 
Let $X' \to \Spec(A')$ be a deformation of $X$ over $A'$. 
Then the inclusion~\eqref{H1-HH2} takes $\kappa(X')$ to $\kappa(\Dperf(X'))$. 
\end{enumerate}
\end{lemma}

\begin{proof}
In \cite[\S16.6]{SAG} the construction of a class $\kappa(\cC') \in \HH^2(\cC/S, I)$ is given in the case 
$A = I = k$ are fields, 
but the same construction works in our more general setting and can be checked to satisfy the stated properties. 

More concretely, in this paper we shall only need the class $\kappa(\cC')$ for $\cC \hookrightarrow \Dperf(X)$ a semiorthogonal component of 
a scheme $X$ smooth over $A$ with $\dim(X/A)!$ invertible on $A$, 
and $\cC' \hookrightarrow \Dperf(X')$ 
a semiorthogonal component of a deformation of $X$ over $A'$. 
In this setting, the class $\kappa(\cC')$ can be defined by stipulating that properties~\eqref{kappa1} and~\eqref{kappa2} hold. 
Namely, we define $\kappa(\Dperf(X'))$ as the image of $\kappa(X')$ under the map~\eqref{H1-HH2}, 
and define $\kappa(\cC')$ as the image of $\kappa(\Dperf(X'))$ under the map $\HH^2(\Dperf(X)/A, I) \to \HH^2(\cC/A, I)$. 
\end{proof} 

\begin{remark}
In contrast to the geometric situation, in the setting of Lemma~\ref{lemma-KS} the 
set of isomorphism classes of deformations classes of $\cC$ over $A'$ is 
not necessarily a torsor under $\HH^2(\cC/A, I)$, cf. \cite[Remark 16.6.7.6 and Theorem 16.6.10.2]{SAG}. 
\end{remark} 

We can also describe the deformation theory of objects along a deformation of a category. 
If $A' \to A$ is a square-zero extension of rings, $\cC$ is an $A$-linear category, $\cC'$ is a deformation of
$\cC$ over $A'$, and $E \in \cC$ is an object, then a \emph{deformation of $E$ to $\cC'$} is an object $E' \in \cC'$ equipped with an equivalence $E'_A \simeq E \in \cC$ (where we have used the given identification $\cC'_A \simeq \cC$). 

\begin{lemma}
\label{lemma-defE}
Let $0 \to I \to A' \to A \to 0$ be a square-zero extension of rings. 
Let $\cC$ be an $A$-linear category, $\cC'$ a deformation of
$\cC$ over $A'$, and $E \in \cC$ an object. 
Then there is an obstruction class 
\begin{equation*}
\omega(E) \in \Ext^2_A(E, E \otimes I) 
\end{equation*}  
with the following properties: 
\begin{enumerate}
\item 
\label{omega1}
$\omega(E)$ vanishes if and only if a deformation of $E$ to $\cC'$ exists, 
in which case the set of isomorphism classes of deformations of 
$E$ to $\cC'$ forms a torsor under $\Ext^1_A(E, E \otimes I)$. 

\item 
\label{omega2}
Assume the extension $A' \to A$ is trivial, so that by Lemma~\ref{lemma-KS} 
we have a Kodaira--Spencer class $\kappa(\cC') \in \HH^2(\cC/A, I)$, which by the definition 
of Hochschild cohomology corresponds to a natural transformation $\id_{\cC} \to (- \otimes I)[2]$. 
Then writing $\kappa(\cC')(E) \in \Ext_A^2(E, E \otimes I)$ for the class obtained by applying this natural 
transformation to $E$, we have an equality 
\begin{equation*}
\omega(E) = \kappa(\cC')(E). 
\end{equation*} 
\end{enumerate} 
\end{lemma} 

\begin{proof}
Similar to Lemma~\ref{lemma-KS}, the result can be proved using the arguments and results of 
\cite[Chapter 16]{SAG}, cf. \cite[Remark 16.0.0.3]{SAG}. 

More concretely, in this paper we shall only need the result in the case where $\cC \hookrightarrow \Dperf(X)$ 
is a semiorthogonal component for $X$ a noetherian scheme smooth over $A$, 
$\cC' \hookrightarrow \Dperf(X')$ is a semiorthogonal component of a deformation $X'$ of $X$ over $A'$, 
and everything is defined over a base field. 
In this setting, the result can be proved as follows.  
First consider the purely geometric case where $\cC = \Dperf(X)$ and $\cC' = \Dperf(X')$. 
Then \cite[Theorem 3.1.1]{lieblich} gives the existence of a class $\omega(E)$ satisfying property~\eqref{omega1}. 
If the extension $A' \to A$ is trivial, then by 
the main theorem of \cite{huybrechts-thomas} 
we have\footnote{This is where we use our assumption that everything is defined over a base field, see \cite{huybrechts-thomas-erratum}.}  
\begin{equation*}
\omega(E) = ( \id_{E} \otimes \kappa(X') ) \circ A(E), 
\end{equation*}  
where $A(E) \in \Ext^1(E, E \otimes \Omega_{X/\Spec(A)})$ is the Atiyah class of $E$ and $\kappa(X')$ is regarded as an element of $\Ext^1(\Omega_{X/\Spec(A)}, I)$.
One checks  
\begin{equation*}
( \id_{E} \otimes \kappa(X') ) \circ A(E) = \kappa(\Dperf(X))(E) , 
\end{equation*} 
so that~\eqref{omega2} holds. 
Now the case where $\cC \hookrightarrow \Dperf(X)$ and $\cC' \hookrightarrow \Dperf(X')$ are not necessarily equalities follows from two observations: 
an object $E'$ is a deformation of $E$ to $\Dperf(X')$ if and only if $E'$ is a deformation of $E$ to $\cC'$; and we have $\kappa(\Dperf(X))(E) = \kappa(\cC')(E)$ by Lemma~\ref{lemma-KS}\eqref{kappa1}.
\end{proof}


\section{Hodge theory of categories}
\label{section-hodge-theory} 

In this section we explain how to associate natural Hodge structures to $\bC$-linear categories, 
via topological K-theory. 
We use this to formulate several variants of the Hodge conjecture for categories, 
and discuss the relation between these conjectures and their classical counterparts. 
We also prove the results about intermediate Jacobians described in \S\ref{section-intro-IJ}. 

\subsection{Topological K-theory} 
Blanc \cite{blanc} constructed a lax symmetric monoidal topological K-theory functor 
\begin{equation*}
\Ktop \colon \Cat_{\bC} \to \Sp 
\end{equation*} 
from $\bC$-linear categories to the $\infty$-category of spectra. 
The following summarizes the results about this construction that are 
relevant to this paper. 

\begin{theorem}[\cite{blanc}]
{\label{theorem-Ktop}}
{\begin{enumerate}

\item If $\cC = \llangle \cC_1, \dots, \cC_m \rrangle$ is a semiorthogonal decomposition of $\bC$-linear categories, then there is an equivalence 
\begin{equation*}
\Ktop(\cC) \simeq \Ktop(\cC_1) \oplus \cdots \oplus \Ktop(\cC_m)
\end{equation*} 
where the map $\Ktop(\cC) \to \Ktop(\cC_i)$ is induced by the projection functor onto the component~$\cC_i$. 

\item \label{ch-Ktop}
There is a functorial commutative square 
\begin{equation}
\label{diagram-ch-Ktop} 
\vcenter{
\xymatrix{
\rK(\cC) \ar[r]^{\ch} \ar[d] & \HN(\cC) \ar[d] \\ 
\Ktop(\cC) \ar[r]^{\ch^{{\rtop}}} & \HP(\cC) 
}
}
\end{equation} 
where $\rK(\cC)$ denotes the algebraic K-theory of $\cC$, 
$\HN(\cC)$ the negative cyclic homology, and $\HP(\cC)$ the periodic cyclic homology. 

\item \label{Ktop-X-comparison}
If $X$ is a scheme which is separated and of finite type over $\bC$ 
with analytification $X^{\an}$, then there exists a functorial equivalence 
\begin{equation*}
\Ktop(\Dperf(X)) \simeq \Ktop(X^{\an}), 
\end{equation*} 
where the right side denotes the complex K-theory spectrum of the topological space $X^{\an}$. 
Under this equivalence, the left vertical arrow in~\eqref{diagram-ch-Ktop} recovers the usual 
map from algebraic K-theory to topological K-theory, 
and under the identification of $\HP(\Dperf(X))$ with $2$-periodic de Rham cohomology, 
the bottom horizontal arrow in~\eqref{diagram-ch-Ktop} recovers the usual topological Chern character. 
\end{enumerate}}
\end{theorem}

For $\cC \in \Cat_{\bC}$ and an integer $n$, we write 
\begin{equation*}
\Ktop[n](\cC) = \pi_{n} \Ktop(\cC) 
\end{equation*} 
for the $n$-th homotopy group of $\Ktop(\cC)$. 
These groups carry canonical pairings in the proper case. 

\begin{lemma}
\label{lemma-euler-pairing}
Let $\cC$ be a proper $\bC$-linear category. Then for any integer $n$ 
there is a canonical bilinear form $\chi^{\rtop}(-,-) \colon \Ktop[n](\cC) \otimes \Ktop[n](\cC) \to \bZ$, called the Euler pairing, with the following properties: 
\begin{enumerate}
\item \label{euler-pairing-so}
If $\cC = \llangle \cC_1, \dots, \cC_m \rrangle$ is a semiorthogonal decomposition of $\bC$-linear categories, then the inclusions $\Ktop[n](\cC_i) \to \Ktop[n](\cC)$ preserve the Euler pairings, and the direct sum decomposition 
\begin{equation*}
\Ktop[n](\cC) \cong \Ktop[n](\cC_1) \oplus \cdots \oplus \Ktop[n](\cC_m)
\end{equation*} 
is semiorthogonal in the sense that $\chi^{\rtop}(v_i, v_j) = 0$ for $v_i \in \Ktop[n](\cC_i)$, $v_j \in \Ktop[n](\cC_j)$, $i > j$. 

\item If $\chi(-,-) \colon \rK_0(\cC) \otimes \rK_0(\cC) \to \bZ$ denotes the Euler pairing defined by 
\begin{equation*}
\chi(E,F) = \sum_i (-1)^i \dim \Ext_{\bC}^i(E, F) 
\end{equation*} 
for $E , F \in \cC$, 
then the map $\rK_0(\cC) \to \Ktop[0](\cC)$ preserves the Euler pairings. 

\item If $\cC = \Dperf(X)$ for a proper complex variety $X$, then for $v, w \in \Ktop[n](X)$ 
we have 
\begin{equation*}
\chi^{\rtop}(v,w) = p_*(v^{\svee} \otimes w) \in \Ktop[2n](\Spec(\bC) ) \cong \bZ , 
\end{equation*} 
where $p \colon X \to \Spec(\bC)$ is the structure morphism. 
\end{enumerate} 
\end{lemma} 

\begin{proof}
As the functor $\Ktop \colon \Cat_{\bC} \to \Sp$ is lax monoidal, we have a natural map 
\begin{equation*}
\Ktop(\cC^{\op}) \otimes \Ktop(\cC) \to \Ktop(\cC^{\op} \otimes_{\Dperf(\Spec(\bC))} \cC). 
\end{equation*} 
There is a canonical identification $\Ktop(\cC^{\op}) = \Ktop(\cC)$; indeed, this follows from the 
definition of $\Ktop(-)$ in \cite{blanc} and the corresponding identification for algebraic K-theory. 
Therefore, passing to homotopy groups we obtain a map 
\begin{equation*}
\Ktop[n](\cC) \otimes \Ktop[n](\cC) \to \Ktop[2n](\cC^{\op} \otimes_{\Dperf(\Spec(\bC))} \cC). 
\end{equation*} 
As $\cC$ is proper over $\bC$, we have an evaluation functor 
\begin{equation*} 
\cC^{\op} \otimes_{\Dperf(\Spec(\bC))} \cC \to \Dperf(\Spec(\bC)) 
\end{equation*} 
induced by the functor $\cHom_{\bC}(-,-) \colon \cC^{\op} \times \cC \to \Dperf(\Spec(\bC))$. 
Taking the topological K-theory of this functor and composing with the above map, we 
obtain the sought for map 
\begin{equation*}
\chi^{\rtop}(-,-) \colon \Ktop[n](\cC) \otimes \Ktop[n](\cC) \to \Ktop[2n](\Dperf(\Spec(\bC))) \cong \bZ. 
\end{equation*}
All of the claimed properties follow directly from the definition. 
For instance, to show the semiorthogonality claimed in~\eqref{euler-pairing-so}, note that restriction of the pairing to 
$\Ktop[n](\cC_i) \otimes \Ktop[n](\cC_j)$ is induced by the functor 
\begin{equation*} 
\cC_i^{\op} \otimes_{\Dperf(\Spec(\bC))} \cC_j \to \Dperf(\Spec(\bC)) , 
\end{equation*} 
which is in turn induced by the functor $\cHom_{\bC}(-,-) \colon \cC_i^{\op} \times \cC_j \to \Dperf(\Spec(\bC))$; 
but if $i > j$, then this functor vanishes by semiorthogonality. 
\end{proof} 

\begin{remark}
We suspect that if $\cC$ is a smooth proper $\bC$-linear category, 
then the Euler pairing on $\Ktop[n](\cC)$ is nondegenerate. 
As in Lemma~\ref{lemma-mukai}, this (and more) would follow if for instance the functor $\Ktop \colon \Cat_{\bC} \to \Sp$ were monoidal (not only lax monoidal) when restricted to the subcategory of smooth proper $\bC$-linear categories. 
\end{remark}

\begin{proposition}
\label{proposition-Ktop-HS}
Let $\cC \subset \Dperf(X)$ be a $\bC$-linear admissible subcategory, where $X$ is a smooth proper complex variety. 
\begin{enumerate}
\item  For any integer $n$, $\Ktop[n](\cC)$ is a finitely generated abelian group, 
and there is a canonical Hodge structure of weight $-n$ on 
$\Ktop[n](\cC)$ such that there is a canonical isomorphism 
\begin{equation*} 
\gr^p(\Ktop[n](\cC) \otimes \bC) \cong \HH_{n+2p}(\cC) , 
\end{equation*} 
where the left side denotes the $p$-th graded piece of the Hodge filtration.  
\item \label{Ktop-H-comparison}
For $\cC = \Dperf(X)$ the Chern character induces an isomorphism 
\begin{equation*}
\Ktop[n](\Dperf(X)) \otimes \bQ \cong \bigoplus_{k \in \bZ} \rH^{2k - n}(X, \bQ)(k) 
\end{equation*} 
of rational Hodge structures, where $\rH^{2k - n}(X, \bQ)(k)$ denotes (the Tate twist by $k$ of) 
the Betti cohomology of $X$.  
\end{enumerate} 
\end{proposition} 

\begin{remark}
\label{remark-Ktop-HS} 
The existence of an admissible embedding $\cC \subset \Dperf(X)$ in Proposition~\ref{proposition-Ktop-HS} allows us in the proof below to leverage deep results about the Hodge theory of varieties into statements for $\cC$. 
We conjecture, however, that Proposition~\ref{proposition-Ktop-HS} remains true for any smooth proper $\bC$-linear category $\cC$, without assuming the existence of an embedding. 
\end{remark} 

\begin{proof}
First note that $\Ktop[n](\cC)$ is a summand of the finitely generated abelian group $\Ktop[n](X^{\an})$, and hence 
finitely generated. 
The noncommutative Hodge-to-de Rham spectral sequence 
and its degeneration \cite{kaledin1, kaledin2, akhil-degeneration} gives a canonical filtration of 
$\HP_n(\cC)$ whose $p$-th graded piece is $\HH_{n+2p}(\cC)$. 
Consider the Chern character map $\Ktop[n](\cC) \otimes \bC \to \HP_n(\cC)$ from Theorem~\ref{theorem-Ktop}. 
We claim that this map is an isomorphism for $\cC$ as in the proposition, 
and the above filtration provides the desired Hodge structure on $\cC$. 
This claim is preserved under passing to semiorthogonal components, 
so we may assume that $\cC = \Dperf(X)$. 

In this case, it is well-known that the Chern character indeed provides an isomorphism 
\begin{equation*}
\Ktop[n](\Dperf(X)) \otimes \bQ \cong \bigoplus_{k \in \bZ} \rH^{2k - n}(X, \bQ) 
\end{equation*} 
of abelian groups. 
Recall \cite{weibel} 
we have an identification $\HP_n(\Dperf(X)) \cong \bigoplus_{k \in \bZ} \rH_{\dR}^{2k-n}(X)$ with $2$-periodic de Rham cohomology, 
under which the noncommutative Hodge-to-de Rham filtration agrees with the $2$-periodic Hodge-to-de Rham filtration, i.e. 
the filtration corresponding to the Hodge structure on $\bigoplus_{k \in \bZ} \rH^{2k - n}(X, \bQ)(k)$ under 
the comparison isomorphism $\rH^{2k - n}(X, \bC) \cong \rH_{\dR}^{2k-n}(X)$. 
We conclude that $\Ktop[n](\Dperf(X)) \otimes \bC \xrightarrow{\sim} \HP_n(\Dperf(X))$ is an isomorphism, 
the noncommutative Hodge-to-de Rham filtration defines a Hodge structure of weight $-n$, 
and the above  isomorphism of abelian groups provided by the Chern character is in fact an isomorphism of rational Hodge structures. 
\end{proof} 

We will need a generalization of Proposition~\ref{proposition-Ktop-HS} to families of categories. 
This relies on a relative version of Blanc's topological K-theory, due to Moulinos \cite{moulinos}. 
Namely, for a scheme $S$ over $\bC$, 
Moulinos constructs a functor 
\begin{equation*} 
\Ktop(-/S) \colon \Cat_S \to \Shv_{\Sp}(S^{\an}) 
\end{equation*} 
from $S$-linear categories to the $\infty$-category of sheaves of spectra on 
the analytification $S^{\an}$. 

\begin{theorem}[\cite{moulinos}] 
{\label{theorem-KtopS}}

\begin{enumerate}
\item If $S = \Spec(\bC)$, there is an equivalence $\Ktop(-/S) \simeq \Ktop(-)$. 

\item If $\cC = \llangle \cC_1, \dots, \cC_m \rrangle$ is a semiorthogonal decomposition of $S$-linear categories, then there is an equivalence 
\begin{equation*}
\Ktop(\cC/S) \simeq \Ktop(\cC_1/S) \oplus \cdots \oplus \Ktop(\cC_m/S) 
\end{equation*} 
where the map $\Ktop(\cC/S) \to \Ktop(\cC_i/S)$ is induced by the projection functor onto the component~$\cC_i$. 

\item \label{Ktop-sheaf}
If $f \colon X \to S$ is a proper morphism of complex varieties 
and $f^{\an} \colon X^{\an} \to S^{\an}$ is its analytification, 
then $\Ktop(\Dperf(X)/S)$ is the sheaf of spectra on $S^{\an}$ given by the formula $U \mapsto \Ktop((f^{\an})^{-1}(U))$.  

\end{enumerate} 
\end{theorem} 

For $\cC \in \Cat_{S}$ and an integer $n$, we write  
\begin{equation*}
\Ktop[n](\cC/S) = \pi_{n} \Ktop(\cC/S) 
\end{equation*} 
for the $n$-th homotopy sheaf of $\Ktop(\cC/S)$, which is a sheaf of abelian groups on $S^{\an}$. 

\begin{proposition}
\label{proposition-Ktop-VHS}
Let $\cC \subset \Dperf(X)$ be an $S$-linear admissible subcategory, 
where $f \colon X \to S$ is a smooth proper morphism of complex varieties. 
\begin{enumerate}
\item \label{Ktop-fiber}
For any integer $n$, $\Ktop[n](\cC/S)$ is a local system of finitely generated abelian groups on $S^{\an}$ whose 
fiber over any point $s \in S(\bC)$ is $\Ktop[n](\cC_s)$. 
\item 
$\Ktop[n](\cC/S)$ underlies a canonical variation of Hodge structures of 
weight $-n$ on $S^{\an}$, which fiberwise for $s \in S(\bC)$ recovers the Hodge structure on $\Ktop[n](\cC_s)$ 
from Proposition~\ref{proposition-Ktop-HS}. 
\item \label{KtopS-comparison-VHS}
For $\cC = \Dperf(X)$ there is an isomorphism 
\begin{equation*}
\Ktop[n](\Dperf(X)/S) \otimes \bQ \cong \bigoplus_{k \in \bZ} \rR^{2k-n}f^{\an}_* \underline{\bQ}(k) 
\end{equation*} 
of variations of rational Hodge structures over $S^{\an}$. 

\item \label{relative-euler-pairing}
There is a bilinear form $\chi^{\rtop}(-,-) \colon \Ktop[n](\cC/S) \otimes \Ktop[n](\cC/S) \to \underline{\bZ}$, which fiberwise for $s \in S(\bC)$ recovers the Euler pairing on $\Ktop[n](\cC_s)$ 
from Lemma~\ref{lemma-euler-pairing}. 
\end{enumerate}  
\end{proposition} 

\begin{remark}
\label{remark-Ktop-VHS} 
Similar to Remark~\ref{remark-Ktop-HS}, we conjecture that Proposition~\ref{proposition-Ktop-HS} remains true for any smooth proper $S$-linear category. 
\end{remark} 

\begin{proof}
As in Proposition~\ref{proposition-Ktop-HS}, all of the statements reduce to the case $\cC = \Dperf(X)$, in 
which case they follow from standard results. 
For example, let us explain the details of~\eqref{Ktop-fiber}. 
By Ehresmann's theorem and Theorem~\ref{theorem-KtopS}\eqref{Ktop-sheaf}, 
$\Ktop[n](\Dperf(X)/S)$ is a local system of abelian groups on $S^{\an}$ whose 
fiber over any point $s \in S(\bC)$ is $\Ktop[n](X_s) \simeq \Ktop[n](\Dperf(X_s))$. 
This implies that $\Ktop[n](\cC/S)$ is a local system, being a summand of $\Ktop[n](\Dperf(X)/S)$, 
and by functoriality the fiber of this local system over $s \in S(\bC)$ 
is the summand $\Ktop[n](\cC_s)$ of $\Ktop[n](\Dperf(X_s))$. 
\end{proof}

\subsection{The noncommutative Hodge conjecture and its variants} 
Using the above, we can formulate a natural notion of Hodge classes on a category. 

\begin{definition}
Let $\cC \subset \Dperf(X)$ be a $\bC$-linear admissible subcategory, 
where $X$ is a smooth proper complex variety. 
The \emph{group of integral Hodge classes} $\Hdg(\cC, \bZ)$ 
on $\cC$ is the subgroup of Hodge classes in $\Ktop[0](\cC)$ for the 
Hodge structure given by Proposition~\ref{proposition-Ktop-HS}. 
More explicitly, $\Hdg(\cC, \bZ)$ consists of all classes in $\Ktop[0](\cC)$ which 
map to $\HH_0(\cC)$ under the Hodge decomposition  
\begin{equation*}
\Ktop[0](\cC) \otimes \bC \cong \bigoplus_{p+q = 0} \HH_{p-q}(\cC) . 
\end{equation*} 
The \emph{group of rational Hodge classes} is defined by 
$\Hdg(\cC, \bQ) = \Hdg(\cC, \bZ) \otimes \bQ$. 
We say an element $v \in \Ktop[0](\cC)$ is \emph{algebraic} if it is in the image of $\rK_0(\cC) \to \Ktop[0](\cC)$; similarly, an element $v \in \Ktop[0](\cC) \otimes \bQ$ is \emph{algebraic} if it is in the image of $\rK_0(\cC) \otimes \bQ \to \Ktop[0](\cC) \otimes \bQ$. 
\end{definition} 

By Proposition~\ref{proposition-Ktop-HS}\eqref{Ktop-H-comparison}, 
if $X$ is a smooth proper complex variety, then the Chern character 
identifies $\Hdg(\Dperf(X), \bQ)$ with the 
usual group of rational Hodge classes $\Hdg^*(X, \bQ)$, i.e. the group of Hodge classes for 
$\bigoplus_{k \in \bZ} \rH^{2k}(X, \bQ)(k)$. 
Recall that the cycle class map from the Chow ring 
\begin{equation*}
\CH^*(X) \otimes \bQ \to \rH^*(X, \bQ) 
\end{equation*} 
factors through $\Hdg^*(X, \bQ)$, and the usual Hodge conjecture predicts that 
this map surjects onto $\Hdg^*(X, \bQ)$. 
Since the Chern character gives an isomorphism 
\begin{equation*} 
\rK_0(\Dperf(X)) \otimes \bQ \cong \CH^*(X) \otimes \bQ, 
\end{equation*} 
we conclude that the map 
\begin{equation*}
\rK_0(\Dperf(X)) \to \Ktop[0](\Dperf(X)) 
\end{equation*}
factors through $\Hdg(\Dperf(X), \bZ)$, and the usual Hodge conjecture is equivalent to 
the surjectivity of the map $\rK_0(\Dperf(X)) \otimes \bQ \to \Hdg(\Dperf(X), \bQ)$. 
Now using additivity under semiorthogonal decompositions of all the invariants in sight 
leads to the following lemma and conjecture.  

\begin{lemma}
Let $\cC \subset \Dperf(X)$ be a $\bC$-linear admissible subcategory, 
where $X$ is a smooth proper complex variety. 
Then the map $\rK_0(\cC) \to \Ktop[0](\cC)$ factors through $\Hdg(\cC, \bZ) \subset \Ktop[0](\cC)$. 
\end{lemma}

\begin{conjecture}[Noncommutative Hodge conjecture] 
Let $\cC \subset \Dperf(X)$ be a $\bC$-linear admissible subcategory, 
where $X$ is a smooth proper complex variety. 
Then the map 
\begin{equation*}
\rK_0(\cC) \otimes \bQ \to \Hdg(\cC, \bQ) 
\end{equation*} 
is surjective. 
\end{conjecture} 

We record the following observation from above. 
\begin{lemma}
\label{lemma-HC-iff-NCHC}
Let $X$ be a smooth proper complex variety. 
Then the Hodge conjecture holds for $X$ if and only if the 
Hodge conjecture holds for $\Dperf(X)$. 
\end{lemma}

There is also an obvious integral variant of the Hodge conjecture for categories. 
\begin{conjecture}[Noncommutative integral Hodge conjecture] 
\label{conjecture-NCIHC}
Let $\cC \subset \Dperf(X)$ be a $\bC$-linear admissible subcategory, 
where $X$ is a smooth proper complex variety. 
Then the map 
\begin{equation*}
\rK_0(\cC) \to \Hdg(\cC, \bZ) 
\end{equation*} 
is surjective. 
\end{conjecture} 

\begin{remark}
As we explain in Example~\ref{example-NCIHC-false} below, 
Conjecture~\ref{conjecture-NCIHC} is false in general. 
Nonetheless, we call it a ``conjecture'' in keeping with 
similar terminology for the (known to be false) integral Hodge conjecture 
for varieties.  
\end{remark} 

The integral Hodge conjectures for varieties and categories are closely related, but 
not so simply as in the rational case. 
The result can be conveniently formulated in terms of Voisin groups. 
Recall from \S\ref{section-introduction} that for a smooth proper complex variety $X$, 
the degree $n$ Voisin group $\Voi^n(X)$ is defined as the cokernel of the cycle 
class map $\CH^n(X) \to \Hdg^n(X, \bZ)$.

\begin{definition}
Let $\cC \subset \Dperf(X)$ be a $\bC$-linear admissible subcategory, 
where $X$ is a smooth proper complex variety. 
The \emph{Voisin group} of $\cC$ is the cokernel 
\begin{equation*}
\Voi(\cC) = \coker(\rK_0(\cC) \to \Hdg(\cC, \bZ)). 
\end{equation*} 
\end{definition} 
Note that the integral Hodge conjecture holds for $\cC$ if and only if $\Voi(\cC) = 0$. 

\begin{proposition}
\label{proposition-IHC-vs-NCIHC}
Let $X$ be a smooth proper complex variety. 
Assume that $\rH^*(X, \bZ)$ is torsion free. 
\begin{enumerate}
\item \label{IHC-implies-NCIHC}
If the integral Hodge conjecture holds in all degrees for $X$, then the 
integral Hodge conjecture holds for $\Dperf(X)$. 
\item \label{NCIHC-implies-IHC}
Assume further that for some integer $n$, the cohomology 
$\rH^{2m}(X, \bZ)$ is of Tate type for all $m > n$, i.e. 
$\rH^{2m}(X, \bC) = \rH^{m,m}(X)$ for $m > n$. 
If $\Voi(\Dperf(X))$ is $d$-torsion for some integer $d$, then 
the $\Voi^m(X)$ is $d (m-1)!$-torsion for all $m \geq n$. 
In particular, if $n = 2$ and the integral Hodge conjecture holds for $\Dperf(X)$, then 
the integral Hodge conjecture in degree $2$ holds for $X$.  
\end{enumerate}
\end{proposition}

\begin{proof}
For the proof we will need the following properties, which 
hold by \cite[\S2.5]{atiyah-hirzebruch} due to our assumption that $\rH^*(X, \bZ)$ is torsion free: 
\begin{enumerate}
\item \label{Ktop-tf} $\Ktop[0](X)$ is torsion free and $\ch \colon \Ktop[0](X) \to \rH^{\mathrm{even}}(X, \bQ)$ is injective. 
\item \label{cha-integral} 
For any $v \in \Ktop[0](X)$ the leading term of $\ch(v)$ is integral, i.e. if 
$\ch(v) = \alpha_i + \alpha_{i+1} + \cdots$ with $\alpha_j \in \rH^{2j}(X, \bQ)$ 
then $\alpha_i \in \rH^{2i}(X, \bZ)$. 
\item \label{lift-to-Ktop} For any $\alpha_i \in \rH^{2i}(X, \bZ)$ there exists $v \in \Ktop[0](X)$ 
such that the leading term of $\ch(v)$ is~$\alpha_i$. 
\end{enumerate} 
(The analogous assertions relating $\Ktop[1](X)$ and the odd cohomology of $X$ are also true, but we will not need this.) 

Now assume that the integral Hodge conjecture holds in all degrees for $X$. 
We must show that any $v \in \Hdg(\Dperf(X), \bZ)$ 
is in the image of $\rK_0(\Dperf(X))$. 
Write $\ch(v) = \alpha_i + \alpha_{i+1} + \cdots$ as above. 
Then $\alpha_i$ is a Hodge class by Proposition~\ref{proposition-Ktop-HS}\eqref{Ktop-H-comparison} and integral by property~\eqref{cha-integral} above, i.e. $\alpha_i \in \Hdg^{i}(X, \bZ)$. 
Therefore, by assumption there are closed 
subvarieties $Z_k \subset X$ of codimension $i$ and integers $c_k \in \bZ$ 
such that $\alpha_i$ is the cycle class of $\sum c_k Z_k$. 
Replacing $v$ by $v - \sum c_k [\cO_{Z_k}]$, we may thus assume $\alpha_i = 0$. 
Continuing in this way, we may assume that $\ch(v) = 0$.
But then $v = 0$ by property~\eqref{Ktop-tf} above, so we are done. 
This proves part~\eqref{IHC-implies-NCIHC} of the proposition. 

Now assume that the cohomological condition in part~\eqref{NCIHC-implies-IHC} of the proposition holds, and that $\Voi(\Dperf(X))$ is $d$-torsion. 
Let $m \geq n$ and $\alpha_m \in \Hdg^m(X,\bZ)$. 
By property~\eqref{lift-to-Ktop} above we may choose a class $v \in \Ktop[0](X)$ such that 
$\ch(v) = \alpha_m + \alpha_{m+1} + \cdots$ where $\alpha_i \in \rH^{2i}(X, \bQ)$. 
By assumption $\alpha_m$ is a Hodge class, and so is $\alpha_i$ for $i > m$ because 
$\rH^{2i}(X, \bZ)$ is of Tate type. 
Thus $v \in \Hdg(\Dperf(X), \bZ)$ is a Hodge class by Proposition~\ref{proposition-Ktop-HS}\eqref{Ktop-H-comparison}. 
Therefore, by assumption there is an object $E \in \Dperf(X)$ whose class in 
$\Ktop[0](X)$ is $dv$, and so $\ch(E) = d\alpha_{m} + d\alpha_{m+1} + \cdots$. 
By the standard formula for the Chern character in terms of Chern classes, 
the vanishing of $\ch_i(E)$ for $i < m$ 
implies that $d\alpha_m = \frac{(-1)^{m-1}}{(m-1)!} c_m(E)$ 
in $\rH^{2m}(X, \bQ)$. 
By torsion freeness of $\rH^{2m}(X, \bZ)$, this is equivalent to 
$d(m-1)! \alpha_m = (-1)^{m-1}c_m(E)$ in $\rH^{2m}(X, \bZ)$. 
This proves that $d(m-1)!$ kills the class of $\alpha_m$ in the cokernel of $\CH^m(X) \to \Hdg^m(X, \bZ)$, as 
required. 
\end{proof} 

\begin{corollary}
\label{corollary-NCIHC-surface}
Let $X$ be a smooth proper complex variety with $\dim(X) \leq 2$, and assume 
$\rH^*(X, \bZ)$ is torsion free in case $\dim(X) = 2$. 
Then the integral Hodge conjecture holds for $\Dperf(X)$. 
\end{corollary} 

\begin{proof}
Follows from Proposition~\ref{proposition-IHC-vs-NCIHC}\eqref{IHC-implies-NCIHC} because the 
integral Hodge conjecture holds for varieties of dimension at most $2$. 
\end{proof} 

\begin{corollary}
\label{corollary-NCIHC-threefold}
Let $X$ be a smooth proper complex threefold 
with $\rH^*(X, \bZ)$ torsion free.
Then the integral Hodge conjecture holds for $X$ if and only if the integral Hodge conjecture holds for $\Dperf(X)$. 
\end{corollary}

\begin{proof}
For a threefold the integral Hodge conjecture always holds in degrees $n=0,1,3$, so the only interesting case is $n = 2$. 
Thus the result follows from Proposition~\ref{proposition-IHC-vs-NCIHC}. 
\end{proof} 

\begin{example}
\label{example-NCIHC-false}
Let $X \subset \bP^4$ be a very general complex hypersurface of degree divisible by $p^3$ 
for an integer $p$ coprime to $6$. 
Then Koll\'{a}r showed that the integral Hodge conjecture in degree $2$ fails for $X$ \cite{trento}. 
By Corollary~\ref{corollary-NCIHC-threefold} we conclude that 
the integral Hodge conjecture also fails for $\Dperf(X)$. 
\end{example} 

The (integral) Hodge conjecture for categories behaves well under semiorthogonal decompositions. 
This will be important in our applications to the integral Hodge conjecture for varieties with CY2 semiorthogonal components. 

\begin{lemma}
\label{lemma-IHC-sod}
Let $\cC \subset \Dperf(X)$ be a $\bC$-linear admissible subcategory, 
where $X$ is a smooth proper complex variety. 
Let $\cC = \llangle \cC_1, \dots, \cC_m \rrangle$ be a $\bC$-linear semiorthogonal decomposition. 
Then there is an isomorphism of Voisin groups 
\begin{equation*}
\Voi(\cC) \cong \Voi(\cC_1) \oplus \cdots \oplus \Voi(\cC_{m}) . 
\end{equation*} 
In particular, the (integral) Hodge conjecture holds for $\cC$ if and only if the (integral) Hodge conjecture holds for all of the semiorthogonal components $\cC_1, \dots, \cC_m$. 
\end{lemma}

\begin{proof}
Follows immediately from the fact that all of the invariants involved in the definition of $\Voi(\cC)$ 
are additive under semiorthogonal decompositions. 
\end{proof} 

We can also formulate a version of the variational Hodge conjecture for categories. 
\begin{conjecture}[Noncommutative variational Hodge conjecture]
\label{conjecture-VHC}
Let $\cC \subset \Dperf(X)$ be an $S$-linear admissible subcategory, 
where $f \colon X \to S$ is a smooth proper morphism of complex varieties. 
Let $\vphi$ be a section of the local system $\Ktop[0](\cC/S) \otimes \bQ$ of $\bQ$-vector spaces on $S^{\an}$. 
Assume there exists a complex point $0 \in S(\bC)$ such that the fiber 
$\vphi_0 \in \Ktop[0](\cC_0) \otimes \bQ$ is algebraic. 
Then $\vphi_s$ is algebraic for every $s \in S(\bC)$. 
\end{conjecture}

Note that as in Lemma~\ref{lemma-HC-iff-NCHC}, 
for $\cC = \Dperf(X)$ the noncommutative variational Hodge conjecture is equivalent to 
the usual variational Hodge conjecture. 
In general, this conjecture is extremely difficult. 
One of the main results of this paper, Theorem~\ref{theorem-VHC}, is an 
integral version of the noncommutative variational Hodge conjecture for  
families of CY2 categories. 

We note that using deep known results for varieties, it is easy to prove 
the statement obtained by replacing ``algebraic'' with ``Hodge'' in the 
noncommutative variational Hodge conjecture. 

\begin{lemma}
\label{lemma-invariant-cycles}
Let $\cC \subset \Dperf(X)$ be an $S$-linear admissible subcategory, 
where $f \colon X \to S$ is a smooth proper morphism of complex varieties. 
Let $\vphi$ be a section of the local system $\Ktop[0](\cC/S) \otimes \bQ$ of $\bQ$-vector spaces on $S^{\an}$. 
Assume there exists a complex point $0 \in S(\bC)$ such that 
the fiber $\vphi_0 \in \Ktop[0](\cC_0) \otimes \bQ$ is a Hodge class. 
Then $\vphi_s$ is a Hodge class for every $s \in S(\bC)$. 
\end{lemma}

\begin{proof}
As in earlier arguments, we may reduce to the case where $\cC = \Dperf(X)$. 
Then in view of the isomorphism of Proposition~\ref{proposition-Ktop-VHS}\eqref{KtopS-comparison-VHS}, this is a well-known consequence of Deligne's global invariant cycle theorem, 
see \cite[Proposition 11.3.5]{charles}. 
\end{proof}

\subsection{Odd degree cohomology and intermediate Jacobians} 
\label{section-IJ}
The following gives conditions under which the odd topological K-theory of categories recovers the odd integral cohomology of varieties. 
For an abelian group $A$, we write $A_{\tf}$ for the quotient by its torsion subgroup. 
\begin{proposition}
\label{proposition-odd-cohomology} 
Let $X$ be a smooth proper complex variety of odd dimension $n$, 
such that $\rH^k(X, \bZ) = 0$ for all odd $k < n$. 
Let $\Dperf(X) = \llangle \cC_1, \dots, \cC_m \rrangle$ be a 
semiorthogonal decomposition. 
Then the Chern character induces an isometry of weight~$n$ Hodge structures 
\begin{equation*}
\ch \colon \Ktop[-n](\cC_1)_{\tf} \oplus \cdots \oplus \Ktop[-n](\cC_m)_{\tf} 
\xrightarrow{\sim} \rH^n(X, \bZ)_{\tf} 
\end{equation*} 
where the left side is the orthogonal sum of the Hodge structures $\Ktop[-n](\cC_i)_{\tf}$ of Proposition~\ref{proposition-Ktop-HS} equipped with their Euler pairings of 
Lemma~\ref{lemma-euler-pairing}, 
and the right side is equipped with its standard Hodge structure and pairing. 
If moreover $\rH^k(X, \bZ) = 0$ for all odd $k \neq n$, then the above isomorphism holds before 
quotienting by torsion, i.e. 
\begin{equation*}
\Ktop[-n](\cC_1) \oplus \cdots \oplus \Ktop[-n](\cC_m) 
\xrightarrow{\sim} \rH^n(X, \bZ) .  
\end{equation*} 
\end{proposition} 

\begin{proof}
If the decomposition of $\Dperf(X)$ is trivial, i.e. $m = 1$, 
then the result holds by (the proof of) \cite[Proposition 2.1 and Remark 2.3]{ottem-rennemo}. 
The general case follows from additivity of the invariants involved under semiorthogonal decompositions. 
The only point which requires explanation is that the direct sum decomposition is orthogonal for the Euler pairing. 
By Lemma~\ref{lemma-euler-pairing} the sum is semiorthogonal. 
By the $m = 1$ case we have a Hodge isometry $\Ktop[-n](\Dperf(X))_{\tf} \cong \rH^n(X, \bZ)_{\tf}$ where the left side is equipped with the Euler pairing and the right side with the usual pairing on cohomology. 
But the pairing on $\rH^n(X, \bZ)_{\tf}$ is anti-symmetric, so the same is true of the Euler pairing, and orthogonality of the direct sum decomposition follows from 
semiorthogonality. 
\end{proof} 

Recall that if $H$ is a Hodge structure of odd weight $n$, then  
its \emph{intermediate Jacobian} is  
\begin{equation*}
J(H) = \frac{H_{\bC}}{F^{\frac{n+1}{2}}(H_{\bC}) \oplus H_{\bZ} }, 
\end{equation*} 
where $H_{\bZ}$ is the underlying abelian group of $H$, 
$H_{\bC}$ is its complexification, and 
$F^{\bullet}(H_{\bC})$ is the Hodge filtration. 
In general $J(H)$ only has the structure of a complex torus, 
but if $H$ is polarized and its Hodge decomposition only has two terms 
(i.e. $H_{\bC} = H^{p,q} \oplus H^{q,p}$ for some $p,q$),  
then $J(H)$ is a principally polarized abelian variety. 
If $X$ is a smooth proper complex variety, 
for an odd integer $k$ we write $J^k(X)$ for the intermediate Jacobian 
associated to the Hodge structure $\rH^k(X, \bZ)$. 

\begin{definition}
\label{definition-IJC}
Let $\cC \subset \Dperf(X)$ be a $\bC$-linear admissible subcategory, where 
$X$ is a smooth proper complex variety. 
The \emph{intermediate Jacobian} of $\cC$ is the complex 
torus 
\begin{equation*}
J(\cC) = J(\Ktop[1](\cC))
\end{equation*} 
given by the intermediate Jacobian of the weight $-1$ Hodge structure 
$\Ktop[1](\cC)$. 
\end{definition} 

\begin{remark} 
For any odd integer $n$, we can also consider the intermediate Jacobian of 
$\Ktop[n](\cC)$. 
However, these are all isomorphic for varying odd $n$, 
because by 2-periodicity of topological K-theory the Hodge structures 
$\Ktop[n](\cC)$ are Tate twists of each other. 
\end{remark} 

Using Proposition~\ref{proposition-odd-cohomology} we can prove 
Theorem~\ref{theorem-IJ} and Corollary~\ref{corollary-IJ-threefold} announced in \S\ref{section-intro-IJ}. 

\begin{proof}[Proof of Theorem~\ref{theorem-IJ}]
By Proposition~\ref{proposition-odd-cohomology} and our assumptions 
on the cohomology of $X$ and $Y_i$, we have isomorphisms 
$J(\Dperf(X)) \cong J^n(X)$ and 
$J(\Dperf(Y_i)) \cong J^{n_i}(Y_i)$ if $n_i$ is odd. 
Moreover, Proposition~\ref{proposition-Ktop-HS}\eqref{Ktop-H-comparison} and our cohomological assumption imply that if $n_i$ is even then $\Ktop[1](\Dperf(Y_i))$ is torsion, so $J(\Dperf(Y_i)) = 0$. 
Thus by Proposition~\ref{proposition-odd-cohomology} applied to the 
semiorthogonal decomposition of $\Dperf(X)$, we 
obtain an isomorphism of complex tori 
\begin{equation}
\label{equation-JnX}
J^n(X) \cong \bigoplus_{n_i \, \mathrm{odd}} J^{n_i}(Y_i) . 
\end{equation}  
Under the further assumption that $\rH^{n_i}(Y_i, \bC) = \rH^{p_i, q_i}(Y_i) \oplus \rH^{q_i, p_i}(Y_i)$, the HKR isomorphism shows that 
$\rH^n(X, \bC) = \rH^{p,q}(X) \oplus \rH^{q,p}(X)$, where $p_i, q_i, p, q$ are as in the statement of Theorem~\ref{theorem-IJ}. 
The fact that the isomorphisms of Proposition~\ref{proposition-odd-cohomology} respect the pairings on each side then implies that the above isomorphism~\eqref{equation-JnX} respects the principal polarizations on each side. 
\end{proof} 

\begin{proof}[Proof of Corollary~\ref{corollary-IJ-threefold}] 
Note that the category $\llangle E_j \rrangle \subset \Dperf(X)$ generated by $E_j$ is equivalent to $\Dperf(\Spec(\bC))$. 
Hence the first part of Corollary~\ref{corollary-IJ-threefold} is just a special case of Theorem~\ref{theorem-IJ}. 
Under the further assumption that $\rH^5(X, \bZ) = 0$, we can apply the second part of Proposition~\ref{proposition-odd-cohomology} (and the $2$-periodicity of $\Ktop[n](-)$) to conclude there are isomorphisms 
\begin{align*}
\rH^3(X, \bZ) & \cong \Ktop[-3](\Dperf(X))  \\ 
& \cong \Ktop[-3](\Dperf(C_1)) \oplus \cdots \oplus \Ktop[-3](\Dperf(C_r)) \\ 
& \cong \rH^1(C_1, \bZ) \oplus \cdots \oplus \rH^1(C_r, \bZ). 
\end{align*} 
In particular, $\rH^3(X, \bZ)_{\tors} = 0$. 
\end{proof} 

\begin{example}
\label{example-XL-YL}
Let $V_6$ be a $6$-dimensional vector space, and consider the Pl\"{u}cker embedded 
Grassmannian $\Gr(2,V_6) \subset \bP(\wedge^2V_6)$ and the Pfaffian 
cubic hypersurface $\Pf(4, V_6^{\svee}) \subset \bP(\wedge^2 V_6^{\svee})$ 
parameterizing forms of rank at most $4$. 
Let $L \subset \wedge^2V_6$ be a codimension $3$ linear subspace, and let 
$L^{\perp} = \ker(\wedge^2V_6^{\svee} \to L^{\svee})$ be its orthogonal. 
We assume $L$ is generic so that the intersections 
\begin{equation*}
X_L = \Gr(2, V_6) \cap \bP(L) \qquad \text{and} \qquad 
Y_L = \Pf(4, V_6^{\svee}) \cap \bP(L^{\perp}) 
\end{equation*} 
are smooth of expected dimension, in which case 
$X_L$ is a Fano fivefold and $Y_L$ is an elliptic curve. 
Then there is an isomorphism 
\begin{equation}
\label{JXL-JYL} 
J^5(X_L) \cong J^1(Y_L) 
\end{equation} 
of principally polarized abelian varieties. 
Indeed, by \cite[\S10]{kuznetsov-grassmannian-lines} there is a semiorthogonal decomposition 
of $\Dperf(X_L)$ consisting of $\Dperf(Y_L)$ and exceptional objects, 
and by the Lefschetz hyperplane theorem $\rH^k(X_L, \bZ) = 0$ for all odd $k < 5$, 
so Theorem~\ref{theorem-IJ} gives the result. 
\end{example} 

\begin{remark}
In \cite{bernardara-tabuada} Bernardara and Tabuada give criteria for relating the intermediate Jacobians 
of varieties whose derived categories share a semiorthogonal component, but their criteria are 
quite restrictive and difficult to verify, especially in dimension $5$ and up. 
In particular, they were unable to prove the isomorphism \eqref{JXL-JYL} (see \cite[Remark 1.18]{bernardara-tabuada}).   
\end{remark} 

Homological projective geometry \cite{kuznetsov-HPD, categorical-plucker, categorical-joins, categorical-cones} is a general theory for producing varieties whose derived categories have a semiorthogonal component in common, 
and gives many examples to which our methods apply.
Example~\ref{example-XL-YL} is a simple instance of this. 
To explain a more interesting example, we need some terminology. 

\begin{definition}
\label{definition-GM} 
A \emph{Gushel--Mukai (GM) variety} is a smooth $n$-dimensional intersection 
\begin{equation*}
X = \Cone(\Gr(2,5)) \cap Q, \quad 2 \leq n \leq 6 , 
\end{equation*} 
where $\Cone(\Gr(2,5)) \subset \bP^{10}$ is the projective cone over the Pl\"{u}cker embedded Grassmannian 
$\Gr(2,5) \subset \bP^{9}$ and $Q \subset \bP^{10}$ is a quadric hypersurface in a linear subspace $\bP^{n+4} \subset \bP^{10}$. 
\end{definition} 

The results of Gushel \cite{gushel} and Mukai \cite{mukai} show these varieties coincide with the class of all smooth Fano varieties of Picard number $1$, coindex $3$, and degree $10$ (corresponding to $n \geq 3$), together with the Brill--Noether general polarized K3 surfaces of degree $10$ (corresponding to $n = 2$). Recently, GM varieties have attracted attention because of the rich structure of their birational geometry, Hodge theory, and derived categories \cite{DIM3fold, IM-EPW, DIM4fold, DebKuz:birGM, DebKuz:moduli, DebKuz:periodGM, GM-IJ, KuzPerry:dercatGM, GMstability}. 

In \cite{DebKuz:birGM} Debarre and Kuznetsov classified 
GM varieties in terms of linear algebraic data, 
by constructing for any GM variety $X$ a \emph{Lagrangian data set} $(V_6(X), V_5(X), A(X))$, where 
\begin{itemize} 
\item $V_6(X)$ is a $6$-dimensional vector space, 
\item $V_5(X) \subset V_6(X)$ is a hyperplane, and 
\item $A(X) \subset \wedge^3 V_6(X)$ is a Lagrangian subspace with respect to the wedge product, 
\end{itemize} 
and proving that $X$ is completely determined by its dimension and this data.  
Interestingly, many of the properties of $X$ only depend on $A(X)$. 
To state such a result for intermediate Jacobians, 
recall that if $X_1$ and $X_2$ are GM varieties such that $\dim(X_1) \equiv \dim(X_2) \pmod 2$, 
then they are called \emph{generalized partners} if there exists an isomorphism $V_6(X_1) \cong V_6(X_2)$ identifying $A(X_1) \subset \wedge^3 V_6(X_1)$ with $A(X_2) \subset \wedge^3V_6(X_2)$, or 
\emph{generalized duals} if there exists an isomorphism $V_6(X_1) \cong V_6(X_2)$ identifying $A(X_1) \subset \wedge^3 V_6(X_1)$ with $A(X_2)^{\perp} \subset \wedge^3V_6(X_2)^{\svee}$. 
For a GM variety of odd dimension $n$, we have $\rH^n(X, \bC) = \rH^{p,q}(X) \oplus \rH^{q,p}(X)$ where $p-q =1$ and $\rH^{p,q}(X)$ is $10$-dimensional \cite[Proposition 3.1]{DebKuz:periodGM}, so $J^n(X)$ is a $10$-dimensional principally polarized abelian variety. 

\begin{theorem}[{\cite{GM-IJ}}]
\label{theorem-GM-IJ} 
Let $X_1$ and $X_2$ be GM varieties of odd dimensions $n_1$ and $n_2$ 
which are generalized partners or duals.  
Then there is an isomorphism 
\begin{equation*}
J^{n_1}(X_1) \cong J^{n_2}(X_2) 
\end{equation*} 
of principally polarized abelian varieties. 
\end{theorem}

This is proved in \cite{GM-IJ} by intricate geometric arguments, 
but as we explain now it can be deduced as a consequence of a categorical statement. Recall that by \cite{KuzPerry:dercatGM}, for any GM variety there is a Kuznetsov component $\Ku(X) \subset \Dperf(X)$  
defined by the semiorthogonal decomposition 
\begin{equation}
\label{equation-Ku}
\Dperf(X) = \llangle \Ku(X), \cO_X, \cU_X^{\vee}, \dots, \cO_X(\dim(X)-3), \cU_X^{\vee}(\dim(X)-3) \rrangle  , 
\end{equation} 
where $\cU_X$ and $\cO_X(1)$ denote the pullbacks to $X$ of the rank $2$ tautological subbundle and Pl\"{u}cker line bundle on $\Gr(2,5)$. 

\begin{lemma}
\label{lemma-JKuX}
Let $X$ be a GM variety of odd dimension $n$. 
Then there is an isomorphism 
\begin{equation*}
J^n(X) \cong J(\Ku(X)) 
\end{equation*} 
of principally polarized abelian varieties. 
\end{lemma} 

\begin{proof}
The Lefschetz hyperplane theorem implies $\rH^k(X, \bZ) = 0$ for all odd $k < n$, 
so the result follows from Proposition~\ref{proposition-odd-cohomology}. 
\end{proof} 

\begin{proof}[Proof of Theorem~\ref{theorem-GM-IJ}]
By \cite[Theorem 3.16]{DebKuz:birGM} for a GM variety $X$ of dimension at least~$3$, the associated Lagrangian subspace $A(X)$ 
does not contain decomposable vectors. 
Hence we may apply the duality conjecture \cite[Conjecture 3.7]{KuzPerry:dercatGM} 
proved in \cite[Theorem 1.6]{categorical-cones} to conclude there is an equivalence $\Ku(X_1) \simeq \Ku(X_2)$. Now the result follows from Lemma~\ref{lemma-JKuX}. 
\end{proof}

\begin{remark}
In \cite{DebKuz:periodGM} an analogue of Theorem~\ref{theorem-GM-IJ} is proved for even-dimensional GM varieties, asserting that generalized partners or duals have the same period. 
This can also be reproved categorically by a more elaborate version of the above argument, 
which will appear in \cite{Fano3fold}. 
\end{remark}


\section{CY2 categories} 
\label{section-CY2}

In this section we define CY2 categories and their associated Mukai Hodge structures, 
and survey the known examples of CY2 categories. 
We also give a sample application of the results of \S\ref{section-hodge-theory} to 
torsion orders of Voisin groups (Corollary~\ref{corollary-dumb-IHC}). 

\subsection{Definitions} 

\begin{definition}
\label{definition-CY2}
A \emph{CY2 category over a field $k$} is a $k$-linear category $\cC$ such that: 
\begin{enumerate}
\item \label{CY2-1}
There exists an admissible $k$-linear embedding $\cC \hookrightarrow \Dperf(X)$, where $X$ is a smooth proper variety over $k$. 
\item \label{CY2-2}
The shift functor $[2]$ is a Serre functor for $\cC$ over $k$. 
\item \label{CY2-3} 
The Hochschild cohomology of $\cC$ satisfies $\HH^0(\cC/k) = k$. 
\end{enumerate}
More generally, a \emph{CY2 category over a scheme $S$} is an $S$-linear category $\cC$ such that: 
\begin{enumerate}
\item There exists an admissible $S$-linear embedding $\cC \hookrightarrow \Dperf(X)$, where $X \to S$ is a smooth proper morphism. 
\item For every point $s \colon \Spec(k) \to S$, the fiber $\cC_s$ is a CY2 category over $k$. 
\end{enumerate}
\end{definition} 

\begin{remark}
\label{remark-CY2-smooth-proper} 
Condition~\eqref{CY2-1} in the definition of a CY2 category over a field or scheme 
says that $\cC$ ``comes from geometry''. 
For all of the results in this paper, it would be enough to instead assume that $\cC$ is smooth and proper and the conclusions of Propositions~\ref{proposition-Ktop-HS} and~\ref{proposition-Ktop-VHS} hold, cf. Remarks~\ref{remark-Ktop-HS} and~\ref{remark-Ktop-VHS}. 
On the other hand, condition~\eqref{CY2-1} will be automatic in the examples we consider in the paper.  
\end{remark}

\begin{remark}
\label{remark-CY2} 
Let us explain the motivation for the other conditions appearing in the definition of a CY2 category $\cC$ over $k$. 

Condition~\eqref{CY2-2} --- the most important part of the definition of a CY2 category --- says that from the perspective of Serre duality, 
$\cC$ behaves like the derived category of a smooth proper surface with trivial canonical bundle, i.e. 
a two-dimensional Calabi--Yau variety. 

Condition~\eqref{CY2-3} says that $\cC$ is connected in the sense of \cite{kuznetsov-fractional-CY}. 
The source of this terminology is the observation that for a smooth 
proper variety $X$ over $k$, we have $\HH^0(X/k) = \rH^0(X, \cO_X)$,  so $X$ is connected if and only if $\HH^0(X/k) = k$. 
Note also that by Lemma~\ref{lemma-mukai-serre} and condition~\eqref{CY2-2} 
we have $\HH^0(\cC/k) \cong \HH_2(\cC/k)$, and by Lemma~\ref{lemma-mukai} we have 
$\HH_2(\cC/k) \cong \HH_{-2}(\cC/k)$; thus Condition~\eqref{CY2-3} amounts to $\HH_2(\cC/k)$, or equivalently $\HH_{-2}(\cC/k)$, being a $1$-dimensional $k$-vector space. 
\end{remark} 

\begin{definition}
\label{definition-MHS}
Let $\cC$ be a CY2 category over $\bC$. 
The \emph{Mukai Hodge structure} of $\cC$ is the weight $2$ Hodge structure $\tH(\cC, \bZ) = \Ktop[0](\cC)(-1)$, where $\Ktop[0](\cC)$ 
is endowed with the weight $0$ Hodge structure from Proposition~\ref{proposition-Ktop-HS} and $(-1)$ denotes a Tate twist. 
We equip $\tH(\cC, \bZ)$ with the bilinear form $(-,-) = - \chi^{\rtop}(-,-)$ given by the negative of the Euler pairing. 

More generally, if $\cC$ is a CY2 category over a complex variety $S$, then 
the \emph{Mukai local system} is defined by $\tH(\cC/S, \bZ) = \Ktop[0](\cC/S)(-1)$, 
which by Proposition~\ref{proposition-Ktop-VHS} is equipped with the structure of a variation of Hodge structures 
which fiberwise for $s \in S(\bC)$ recovers the Mukai Hodge structure $\tH(\cC_s, \bZ)$. 
\end{definition} 

The Tate twist in the definition is included for historical reasons, 
see Remark~\ref{remark-MHS-K3} below. 

\subsection{Examples}  
\label{section-CY2-examples}
The main known examples of CY2 categories are as follows. 
For simplicity, we will work over the complex numbers, but all of the examples also work 
over a field of sufficiently large characteristic. 
We focus on the absolute case, but all of the constructions also work in families 
to give examples of CY2 categories over a base scheme, as we explain in a particular 
case in Example~\ref{example-cubic}. 

\begin{example}[K3 and abelian surfaces]
\label{example-surface} 
Let $T$ be a smooth connected proper surface 
with trivial canonical bundle, i.e. a K3 or abelian surface. 
Then $\Dperf(T)$ is a CY2 category. Indeed, condition~\eqref{CY2-1} of 
Definition~\ref{definition-CY2} is obvious, condition~\eqref{CY2-2} holds since $T$ has trivial canonical bundle, 
and condition~\eqref{CY2-3} holds since $T$ is connected (see Remark~\ref{remark-CY2}). 
More generally, if $T$ is equipped with a Brauer class $\alpha \in \Br(T)$, then the twisted 
derived category $\Dperf(T, \alpha)$ is a CY2 category. 
\end{example} 

\begin{remark}
\label{remark-MHS-K3} 
For $T$ as above there is an isomorphism 
$\Ktop[0](T) \cong \rH^{\mathrm{even}}(T, \bZ)$ given by \mbox{$v \mapsto \ch(v) \td(T)^{1/2}$}. 
This isomorphism identifies $\tH(\Dperf(T), \bZ)$, equipped with its weight $2$ Hodge structure and pairing from Definition~\ref{definition-MHS}, with the classically defined Mukai Hodge structure on $\rH^{\mathrm{even}}(T, \bZ)$. 
Similarly, $\tH(\Dperf(T, \alpha), \bZ)$ recovers the usual Mukai Hodge structure in the twisted case. 
\end{remark} 

\begin{example}[Cubic fourfolds]
\label{example-cubic}
Let $X \subset \bP^5$ be a \emph{cubic fourfold}, i.e. a smooth cubic hypersurface. 
The \emph{Kuznetsov component} $\Ku(X) \subset \Dperf(X)$ is the subcategory defined by the semiorthogonal 
decomposition 
\begin{equation*}
\Dperf(X) = \llangle \Ku(X), \cO_X, \cO_X(1), \cO_X(2) \rrangle . 
\end{equation*} 
The category $\Ku(X)$ was introduced by 
Kuznetsov \cite{kuznetsov-cubic} 
who proved that it is an example of CY2 category. 
For very general cubic fourfolds $\Ku(X)$ is not equivalent to the derived category of a (twisted) K3 or abelian surface, so this is a genuinely new example of a CY2 category. 
The category $\Ku(X)$ has close connections to birational geometry, Hodge theory, and hyperk\"{a}hler varieties, 
and has been the subject of many recent works \cite{addington-thomas, huybrechts-cubic, huybrechts-torelli, BLMS, LPZ, BLMNPS}. 

Let us also explain a relative version of the above construction. 
Namely, let $f \colon X \to S$ be a family of cubic fourfolds with $\cO_{X}(1)$ the corresponding relatively ample line bundle. 
Then $f^* \colon \Dperf(S) \to \Dperf(X)$ is fully faithful, and 
the Kuznetsov component $\Ku(X) \subset \Dperf(X)$ is the $S$-linear category 
defined by the $S$-linear semiorthogonal decomposition 
\begin{equation*}
\Dperf(X) = \llangle \Ku(X), f^*\Dperf(S), f^*\Dperf(S) \otimes \cO_X(1), f^* \Dperf(S) \otimes \cO_X(2) \rrangle. 
\end{equation*} 
Note that $\Ku(X)$ fiberwise recovers the Kuznetsov components of the fibers of $f \colon X \to S$, i.e. 
for any point $s \in S$ we have $\Ku(X)_s \simeq \Ku(X_s)$. 
\end{example}

\begin{example}[Gushel--Mukai varieties] 
\label{example-GM}
Let $X$ be a GM variety as in Definition~\ref{definition-GM}. 
Recall that there is a Kuznetsov component $\Ku(X) \subset \Dperf(X)$ as defined in~\eqref{equation-Ku}. 
In~\cite{KuzPerry:dercatGM} it is shown that $\Ku(X)$ is a CY2 category if $\dim(X)$ is even, 
and if $\dim(X) = 4$ or $6$ then for very general 
$X$ the category $\Ku(X)$ is not equivalent to any of the CY2 
categories discussed in Examples~\ref{example-surface} and~\ref{example-cubic} above. 
\end{example}

\begin{example}[Debarre--Voisin varieties] 
\label{example-DV}
A \emph{Debarre--Voisin (DV) variety} is a smooth Pl\"{u}cker hyperplane section $X$ of the Grassmannian 
$\Gr(3,10)$. 
These varieties were originally studied in \cite{debarre-voisin} because of their role in the construction of a certain hyperk\"{a}hler fourfold.  
There is a Kuznetsov component $\Ku(X) \subset \Dperf(X)$ defined by a semiorthogonal decomposition 
\begin{equation*}
\Dperf(X) = \llangle \Ku(X), \cB_X, \cB_X(1), \dots, \cB_X(9) \rrangle. 
\end{equation*} 
Here, $\cB_X \subset \Dperf(X)$ is the subcategory generated by $12$ exceptional objects 
\begin{equation*}
\cB_X = \llangle \Sigma^{\alpha_1, \alpha_2}\cU_X^{\svee} \st 0 \leq \alpha_1 \leq 4, 0 \leq \alpha_2 \leq 2, \alpha_2 \leq \alpha_1 \rrangle ,
\end{equation*} 
where $\cU_X$ denotes the pullback to $X$ of the rank $3$ tautological subbundle on $\Gr(3,10)$, $\Sigma^{\alpha_1, \alpha_2}$ denotes the Schur functor for the Young diagram of type $(\alpha_1, \alpha_2)$, 
and $\cB_X(i)$ denotes the subcategory obtained by tensoring $\cB_X$ with the $i$-th power of the Pl\"{u}cker line bundle restricted to $X$. 
By \cite[Corollary 4.4]{kuznetsov-fractional-CY}, the category $\Ku(X)$ has 
Serre functor $[2]$. 
Moreover, 
a direct computation using the HKR isomorphism shows $\HH^0(\Ku(X)/\bC) = \bC$, so $\Ku(X)$ is a CY2 category. 
Using arguments as in~\cite{KuzPerry:dercatGM}, one can show that 
for very general $X$, the category $\Ku(X)$ is not equivalent to 
any of the CY2 categories discussed in Examples~\ref{example-surface},~\ref{example-cubic}, or~\ref{example-GM}, so this provides yet another new example of a CY2 category. 
\end{example}

\begin{example}
\label{example-dumb-CY2}
Using the results of \cite{kuznetsov-fractional-CY}, it is easy to construct 
other examples of varieties with CY2 categories as a semiorthogonal component, but 
a posteriori one can often show that these CY2 categories reduce to one of the above examples. 
For instance, if $X \subset \bP^3 \times \bP^3 \times \bP^3$ is a smooth divisor of class $H_1 + H_2 + H_3$, where $H_i$ is the hyperplane class on the $i$-th factor, then there is a CY2 category 
$\Ku(X) \subset \Dperf(X)$ defined by the decomposition 
\begin{equation*}
\Dperf(X) = \llangle 
\Ku(X), \pi_{12}^* \Dperf(\bP^3 \times \bP^3), \pi_{12}^* \Dperf(\bP^3 \times \bP^3)(H_3), \pi_{12}^* \Dperf(\bP^3 \times \bP^3)(2H_3) \rrangle
\end{equation*} 
where $\pi_{12} \colon X \to \bP^3 \times \bP^3$ is the projection onto the first two factors. 
However, one can show that $\Ku(X) \simeq \Dperf(T)$, where $T \subset \bP^3 \times \bP^3$ is a K3 
surface given as a complete intersection of $4$ hyperplanes determined by the defining equation of $X$. 

Similarly, by \cite[Corollary 4.2]{kuznetsov-fractional-CY} one obtains an infinite list of weighted projective hypersurfaces whose derived category contains a CY2 category as the orthogonal to a collection of exceptional objects, but it seems that most (and possibly all) of these categories reduce to a known example.  
\end{example} 

The classification of CY2 categories is an important open problem, especially because 
of their role in constructing hyperk\"{a}hler varieties \cite{BLMNPS, GMstability}. 
However, finding new CY2 categories appears to be a difficult problem. 
Besides the above examples, there is conjecturally a new CY2 category arising as a semiorthogonal 
component in the derived category of a so-called K\"{u}chle fourfold \cite{kuznetsov-kuchle}, and 
there is recent work 
that uses Hodge theory to find candidate Fano varieties with CY2 
categories as semiorthogonal components~\cite{FanoK3type}. 
In general, Markman and Mehrotra proved the existence of a family of categories satisfying 
conditions~\eqref{CY2-2} and~\eqref{CY2-3} of Definition~\ref{definition-CY2}
over a Zariski open subset of any moduli space of hyperk\"{a}hler varieties of K3$^{[n]}$ type \cite{markman}; 
we expect these categories also satisfy condition~\eqref{CY2-1}, and thus 
give an infinite series of CY2 categories. 

\begin{remark} 
The categories from Examples~\ref{example-cubic}, \ref{example-GM}, and \ref{example-DV} are all of K3 type in 
the sense that their Hochschild homology agrees with that of a K3 surface. 
In fact, in Examples~\ref{example-cubic} and \ref{example-GM} it is known that for special $X$ 
the category $\Ku(X)$ is equivalent to the derived category of a K3 surface, and conjecturally the same holds in Example~\ref{example-DV}. 
It would be interesting to construct new examples which do not have this property, e.g. 
which have the same Hochschild homology as an abelian surface. 
\end{remark}

To end this section, let us explain how the results of \S\ref{section-hodge-theory} can be applied in Example~\ref{example-dumb-CY2}. 

\begin{corollary}
\label{corollary-dumb-IHC}
Let $X \subset \bP^3 \times \bP^3 \times \bP^3$ be a smooth $(1,1,1)$ divisor. 
Then the Voisin group $\Voi^4(X)$ is $6$-torsion.
\end{corollary} 

\begin{proof}
By the Lefschetz hyperplane theorem, the group $\rH^*(X, \bZ)$ is torsion free and $\rH^{2m}(X, \bZ)$ is of Tate type for $m > 4$. 
By Proposition~\ref{proposition-IHC-vs-NCIHC}\eqref{NCIHC-implies-IHC} 
it thus suffices to prove the integral Hodge conjecture for $\Dperf(X)$. 
But as explained in Example~\ref{example-dumb-CY2}, this category admits 
a semiorthogonal decomposition consisting of the derived category of a 
K3 surface and several copies of $\Dperf(\bP^3 \times \bP^3)$, and 
$\Dperf(\bP^3 \times \bP^3)$ in turn admits a semiorthogonal decomposition 
consisting of copies of the derived category of a point.  
Therefore the result follows from Lemma~\ref{lemma-IHC-sod} and Corollary~\ref{corollary-NCIHC-surface}. 
\end{proof} 

\begin{remark}
It would be interesting to determine whether Corollary~\ref{corollary-dumb-IHC} is optimal, i.e. whether there exists an $X$ 
such that $\Voi^4(X)$ has an order $6$ element.
\end{remark}

Corollary~\ref{corollary-dumb-IHC} illustrates 
the principle that the Hodge conjecture and its variants for a given variety 
can often be reduced to simpler cases via semiorthogonal decompositions. 
Later in \S\ref{section-theorem-proofs} we will use this method to prove 
the integral Hodge conjecture in degree $2$ for cubic and GM fourfolds (Corollary~\ref{corollary-example-IHC}); the key new ingredient needed to handle these examples is the integral 
Hodge conjecture for their Kuznetsov components (Proposition~\ref{proposition-Ku-IHC}).  


\section{Moduli of objects in CY2 categories} 
\label{section-moduli-objects} 

In this section we prove a key result of this paper, Theorem~\ref{theorem-mukai}, 
which asserts the smoothness of suitable relative moduli spaces of objects in families 
of CY2 categories. 

\subsection{Moduli of objects in categories} 

Let $\cC \subset \Dperf(X)$ be an $S$-linear admissible subcategory, 
where $X \to S$ is a smooth proper morphism of complex varieties. 
Recall from \cite[\S9]{BLMNPS} that Lieblich's work \cite{lieblich} implies 
there is an algebraic stack $\cM(\cC/S)$ locally of finite presentation over $S$, parameterizing universally gluable, relatively perfect objects in $\cC$. 
Due to our assumption that $X \to S$ is smooth, the word ``relatively'' can be dropped, i.e. this stack can be defined as follows: for $T \to S$ a scheme over $S$, the $T$-points of $\cM(\cC/S)$ are objects $E \in \cC_{T}$ which for all points $t \in T$ satisfy $\Ext^{<0}_{\kappa(t)}(E_t, E_t) = 0$. 

\begin{remark}
The stack $\cM(\cC/S)$ can also be constructed for $\cC$ smooth and proper without assuming 
the existence of an admissible embedding $\cC \subset \Dperf(X)$, see \cite{toen-moduli-objects}.  
\end{remark}

Note that any object $E \in \cC$ gives rise to a relative class $\vphi_E \in \Gamma(S^{\an}, \Ktop[0](\cC/S))$, 
whose fiber over $s \in S(\bC)$ is the class of $E_{s}$ in $\Ktop[0](\cC_s)$. 

\begin{definition}
For $\vphi \in \Gamma(S^{\an}, \Ktop[0](\cC/S))$,  
we denote by $\cM(\cC/S, \vphi)$ the subfunctor of $\cM(\cC/S)$ parameterizing 
objects with class equal to $\vphi$, and by $s\cM(\cC/S, \vphi)$ the subfunctor of 
$\cM(\cC/S, \vphi)$ parameterizing objects which are simple. 
\end{definition}

\begin{lemma}
\label{lemma-sM-lft}
For $\vphi \in \Gamma(S^{\an}, \Ktop[0](\cC/S))$, 
the functors $\cM(\cC/S, \vphi)$ and $s\cM(\cC/S, \vphi)$ are algebraic stacks locally of finite presentation over $S$, 
and the canonical morphisms 
\begin{equation*}
s\cM(\cC/S, \vphi) \to \cM(\cC/S, \vphi) \to \cM(\cC/S)
\end{equation*} 
are open immersions. Moreover, 
the stack $s\cM(\cC/S, \vphi)$ is a $\bG_m$-gerbe over an algebraic space $s\rM(\cC/S,\vphi)$ locally of finite presentation over $S$. 
\end{lemma} 

\begin{proof}
By \cite[Proposition 9.10]{BLMNPS} it is enough to show that $\cM(\cC/S, \vphi)$ is an open 
substack of $\cM(\cC/S)$. 
This follows from the fact proved in Proposition~\ref{proposition-Ktop-VHS} 
that $\Ktop[0](\cC/S)$ is a local system on $S^{\an}$. 
\end{proof} 

\subsection{Generalized Mukai's theorem} 
We will prove Theorem~\ref{theorem-mukai} after two preparatory lemmas.  
The first will be useful for controlling deformations of simple universally gluable objects in CY2 categories. 
Recall from \S\ref{section-chern-characters} the formalism of Chern classes valued in Hochschild homology.  
 
\begin{lemma}
\label{lemma-simple-ug-CY2}
Let $\cC$ be a CY2 category over a Noetherian $\bQ$-scheme $S$. 
Let $E \in \cC$ be a simple universally gluable object.  
\begin{enumerate}
\item \label{coh-ExtEE}
The cohomology sheaves $\cExt^i_S(E,E)$ of $\cHom_S(E,E)$ are locally free, 
vanish for $i \notin [0,2]$, and are line bundles for $i = 0$ and $i = 2$. 
\item \label{ExtEE-bc}
The formation of the sheaves $\cExt^i_S(E,E)$ commutes with base change, i.e. 
for any morphism $g \colon T \to S$ and $i \in \bZ$ there is a canonical isomorphism 
\begin{equation*}
g^*\cExt^i_S(E,E) \cong \cExt^i_T(E_T,E_T). 
\end{equation*} 
\item \label{chern-simple-ug}
The Chern character map 
\begin{equation*}
\ch \colon \cExt^2_S(E, E) \to \HH_{-2}(\cC/S) 
\end{equation*} 
is an isomorphism. 
\end{enumerate}
\end{lemma} 

\begin{proof}
For any morphism $g \colon T \to S$, we have the base change formula 
\begin{equation}
\label{bc-Hom}
g^* \cHom_S(E, E) \simeq \cHom_T(E_T, E_T) 
\end{equation}
for mapping objects (see \cite[Lemma 2.10]{NCHPD}). 
For a point $i_s \colon \Spec(\kappa(s)) \to S$, this gives for any $i$ an isomorphism 
an isomorphism 
\begin{equation}
\label{Hi-HomEE}
\rH^i(i_s^* \cHom_S(E,E)) \cong \Ext^i_{\kappa(s)}(E_s, E_s). 
\end{equation} 
Since $\rS_{\cC_s} = [2]$ is a Serre functor for $\cC_s$, we have 
\begin{equation*}
\Ext^i_{\kappa(s)}(E_s, E_s) \cong \Ext^{2-i}_{\kappa(s)}(E_s, E_s)^{\svee}. 
\end{equation*} 
We conclude that these groups vanish for $i \notin [0,2]$ because $E$ is universally gluable, 
and are $1$-dimensional for $i = 0,2$ because $E$ is simple. 
In particular, the complex $\cHom_S(E,E)$ is concentrated in degrees $[0,2]$, 
as this is true of $i_s^* \cHom_S(E,E)$ for every point $s$.  

Now we prove~\eqref{chern-simple-ug}. 
As the sheaf $\cExt^2_S(E, E)$ is coherent and 
$\HH_{-2}(\cC/S)$ is finite locally free by Theorem~\ref{theorem-degeneration}, 
to prove $\ch \colon \cExt^2_S(E, E) \to \HH_{-2}(\cC/S)$ is an isomorphism 
it suffices to show that for any point $s \in S$ the (underived) base changed map 
\begin{equation*}
\rH^0(i_s^* \ch) \colon \rH^0(i_s^* \cExt^2_S(E, E))  \to \rH^0(i_s^*\HH_{-2}(\cC/S))
\end{equation*} 
is an isomorphism. 
Note that 
\begin{equation*}
\rH^0(i_s^*\cExt^2_S(E,E)) \cong \rH^2(i_s^*\cHom_S(E,E)) \cong \Ext^2_{\kappa(s)}(E_s, E_s), 
\end{equation*} 
where the first isomorphism holds because, as shown above, $\cExt^2_S(E,E)$ is the top cohomology sheaf of 
$\cHom_S(E,E)$, and the second isomorphism holds by~\eqref{Hi-HomEE}. 
On the other hand, by Theorem~\ref{theorem-degeneration} 
we have 
\begin{equation*}
\rH^0(i_s^*\HH_{-2}(\cC/S)) \cong i_s^*\HH_{-2}(\cC/S) \cong \HH_{-2}(\cC_s/\kappa(s)). 
\end{equation*}
Under these isomorphisms, the map $\rH^0(i_s^* \ch)$ in question is identified with 
the Chern character map 
\begin{equation*}
\ch \colon \Ext^2_{\kappa(s)}(E_s, E_s) \to \HH_{-2}(\cC_s/\kappa(s)). 
\end{equation*} 
Note that the domain and target of this map are $1$-dimensional $\kappa(s)$-vector spaces; 
indeed, for the domain this was observed above, and for the target holds by the definition of a 
CY2 category (see Remark~\ref{remark-CY2}).  
By Lemma~\ref{lemma-mukai-serre} this map is dual to the map 
\begin{equation*}
\Ext_{\kappa(s)}^{-2}(\id_{\cC_s}, \rS_{\cC_s}) \to \Ext_{\kappa(s)}^{-2}(E_s, \rS_{\cC_s}(E_s)) . 
\end{equation*} 
Since $\rS_{\cC_s} = [2]$ this is evidently nonzero, and hence an isomorphism. 
All together this proves~\eqref{chern-simple-ug}, and also shows that $\cExt^2_S(E,E)$ is a 
line bundle, because as observed in the proof $\HH_{-2}(\cC/S)$ is locally free of rank $1$. 

Now we finish the proof of~\eqref{coh-ExtEE}. 
As $E$ is simple, $\cExt^0_S(E,E) \cong \cO_S$ is a line bundle, so it remains only to show 
that $\cExt^1_S(E,E)$ is locally free. 
Since $S$ is Noetherian and $\cExt^1_S(E,E)$ is coherent, by the local criterion for flatness it suffices 
to show $\rH^{-1}(i_s^* \cExt_S^1(E,E)) = 0$ for every point $s \in S$. 
By~\eqref{Hi-HomEE} we have a spectral sequence with $E_2$-page 
\begin{equation*}
E_2^{i,j} = \rH^j(i_s^* \cExt_S^i(E,E)) \implies \Ext_{\kappa(s)}^{i+j}(E_s,E_s). 
\end{equation*} 
By what we have already shown, $E_2^{0,0}$ and $E_2^{2,0}$ are $1$-dimensional, 
$E_2^{0,j}$ and $E_2^{2,j}$ vanish for $j \neq 0$, 
and $\Ext_{\kappa(s)}^{0}(E_s,E_s)$ is $1$-dimensional, 
so the desired vanishing $\rH^{-1}(i_s^* \cExt_S^1(E,E)) = 0$ follows. 

Finally,~\eqref{ExtEE-bc} holds by the base change formula~\eqref{bc-Hom} and the 
freeness of the cohomology sheaves $\cExt^i_S(E,E)$ proved in~\eqref{coh-ExtEE}. 
\end{proof}

The following simple flatness result will be used below to reduce Theorem~\ref{theorem-mukai} to the case where the base $S$ is smooth. 

\begin{lemma}
\label{lemma-flatness} 
Let $f \colon X \to S$ be a locally finite type morphism of schemes, with $S$ reduced and locally noetherian. 
Assume there exists a surjective finite type universally closed morphism $S' \to S$ such that the base change 
$f' \colon X' = X_{S'} \to S'$ is flat. 
Then $f \colon X \to S$ is flat. 
\end{lemma}

\begin{proof}
By the valuative criterion for flatness \cite[11.8.1]{EGA4-3}, it suffices to show that the base change 
$X_T \to T$ is flat for any morphism $g \colon T = \Spec(R) \to S$ from the spectrum of a DVR. 
Let $K$ be the fraction field of $R$. 
As $S' \to S$ is surjective and finite type, there exists a finitely generated field extension $K'$ of $K$ and a morphism $\Spec(K') \to S'$ whose image in $S$ is the image of $\Spec(K) \to S$. 
Let $R'$ be a DVR with fraction field $K'$ which dominates $R$, and let $T' = \Spec(R')$. 
By the valuative criterion for universal closedness \citestacks{05JY} we can find a commutative diagram 
\begin{equation*}
\xymatrix{ 
S' \ar[r] & S  \\ 
T' \ar[r] \ar[u] & T \ar[u]_{g}
}
\end{equation*} 
The base change $X_{T'} \to T'$ is flat by the assumption that $f' \colon X' \to S'$ is flat, 
so since $T' \to T$ is faithfully flat, the base change $X_{T} \to T$ is also flat. 
\end{proof} 

\begin{proof}[Proof of Theorem~\ref{theorem-mukai}]
We will prove that $s\rM(\cC/S,\vphi)$ is smooth over $S$; 
this implies the same for $s\cM(\cC/S, \vphi)$, because it is a 
$\bG_m$-gerbe over $s\rM(\cC/S,\vphi)$. 
We prove the result in several steps. 

\medskip 

\begin{step}{1}
Reduction to the case where the base $S$ is smooth. 
\end{step}

Let $S' \to S$ be a morphism from a complex variety $S'$, 
let $\cC'$ be the base change of $\cC$, and let $\vphi'$ be the pullback of 
the section $\vphi$ to $S'$. 
Then $s\rM(\cC'/S',\vphi') \to S'$ is the base 
change of $s\rM(\cC/S,\vphi) \to S$ along $S' \to S$. 
Hence the theorem in the case where the base is a point implies smoothness 
of the fibers of $s\rM(\cC/S,\vphi) \to S$. 
Similarly, taking a resolution of singularities $S' \to S$, by Lemma~\ref{lemma-flatness} 
we see that the theorem in the case of a smooth base implies flatness of $s\rM(\cC/S,\vphi) \to S$.
Together, this shows that $s\rM(\cC/S,\vphi) \to S$ is smooth. 
Thus from now on we may assume the base $S$ to be smooth. 

\medskip 

\begin{step}{2}
Reduction to proving smoothness of the total space $s\rM(\cC/S,\vphi)$. 
\end{step}

As observed above, the case where the base is a point implies that the closed fibers of $s\rM(\cC/S, \vphi) \to S$ are smooth. 
It follows that these fibers are also of constant dimension $-\chi^{\rtop}(\vphi, \vphi) + 2$, 
where $\chi^{\rtop}(-,-)$ denotes the Euler pairing from Lemma~\ref{proposition-Ktop-VHS}\eqref{relative-euler-pairing}, because this number computes $\dim \Ext^1(E_0, E_0)$ for any object $E_0$ in $s\rM(\cC_0/\Spec(\bC), \vphi_0)$, $0 \in S(\bC)$, as $\cC_0$ is CY2. 
Therefore, the morphism $s\rM(\cC/S,\vphi) \to S$ is smooth, being a locally finite type morphism between smooth spaces whose closed fibers are smooth of constant dimension. 

\medskip 

\begin{step}{3}
\label{step3-smoothness}
Smoothness in terms of deformation functors of objects. 
\end{step} 

Since $s\rM(\cC/S,\vphi)$ is locally of finite type over $\bC$ by Lemma~\ref{lemma-sM-lft}, 
it is smooth if it is formally smooth at any $\bC$-point. 
More precisely, let $E_0$ be a $\bC$-point of $s\rM(\cC/S,\vphi)$ lying over a point 
$0 \in S$. 
Let $\Art_{\bC}$ denote the category of Artinian local $\bC$-algebras with residue field $\bC$, 
and consider the deformation functor 
\begin{equation*}
F \colon \Art_{\bC} \to \Sets
\end{equation*}
of $E_0$, whose value on $A \in \Art_{\bC}$ consists of pairs $(\Spec(A) \to S, E)$ 
where $\Spec(A) \to S$ takes the closed point $p \in \Spec(A)$ to $0 \in S$, 
and $E \in s\rM(\cC/S, \vphi)(A)$ is such that $E_{p} \cong E_0$. 
For simplicity we often write $E \in F(A)$, suppressing the map $\Spec(A) \to S$ 
from the notation. 
Note that in the definition of $F(A)$, the condition $E \in s\rM(\cC/S, \vphi)(A)$ can be replaced with $E \in \cC_A$, 
since then the condition $E_p \cong E_0 \in s\rM(\cC/S, \vphi)(\bC)$  guarantees $E \in s\rM(\cC/S, \vphi)(A)$. 
To prove that $s\rM(\cC/S,\vphi)$ is formally smooth at $E_0$, we must show that 
$F$ is a smooth functor, i.e. for any surjection $A' \to A$ in $\Art_{\bC}$ the map 
$F(A') \to F(A)$ is surjective. 

For this purpose, it will also be useful to consider the deformation functor 
$G \colon \Art_{\bC} \to \Sets$ of the point $0 \in S$, 
whose value on $A \in \Art_{\bC}$ consists of morphisms $\Spec(A) \to S$ taking the closed point to $0$. 
Note that there is a natural morphism of functors $F \to G$. 

\medskip 

\begin{step}{4} 
\label{step4-smoothness}
Let $A' \to A$ be a split square-zero extension in $\Art_{\bC}$, i.e. $A' \cong A[\varepsilon]/(\varepsilon^2)$. 
Then the fiber $\Def_E(A')$ of the map 
\begin{equation*}
F(A') \to G(A') \times_{G(A)} F(A) 
\end{equation*} 
over any point $(\Spec(A') \to S, E)$ is a torsor under $\Ext^1_{A}(E, E)$. 
\end{step} 

By Lemma~\ref{lemma-defE}, the claim is equivalent to the vanishing of the class 
$\kappa(\cC_{A'})(E) \in \Ext^2_A(E, E)$. 
By Lemma~\ref{lemma-simple-ug-CY2}\eqref{chern-simple-ug}, this is in turn equivalent to the vanishing of  
$\ch(\kappa(\cC_{A'})(E)) \in \HH_{-2}(\cC_A/A)$, 
which by Lemma~\ref{lemma-ch-functoriality} is equal to the product $\kappa(\cC_{A'}) \cdot \ch(E)$. 
Finally, this vanishes because by assumption the class of $E$ remains of Hodge type along $S$. 

\medskip 

\begin{step}{5}
The functor $F$ is smooth. 
\end{step} 

For any integer $n \geq 0$ set $A_n = \bC[t]/(t^{n+1})$ and $A'_n = A_n[\varepsilon]/(\varepsilon^2)$.  
By the $T^1$-lifting theorem \cite{T1-kawamata, T1-fantechi}, the functor 
$F$ is smooth if for every $n \geq 0$ the natural map 
\begin{equation*}
F(A'_{n+1}) \to F(A'_n) \times_{F(A_n)} F(A_{n+1}) 
\end{equation*} 
is surjective. 
Note that this map fits into a commutative diagram 
\begin{equation*}
\xymatrix{
F(A'_{n+1}) \ar[r] \ar[d] & F(A'_n) \times_{F(A_n)} F(A_{n+1}) \ar[d] \\ 
G(A'_{n+1}) \ar[r]  & G(A'_n) \times_{G(A_n)} G(A_{n+1}) 
}
\end{equation*} 
where the bottom horizonal arrow is surjective because $0 \in S$ is a smooth point. 
Therefore, it suffices to prove the map 
\begin{equation} 
\label{T1-diagram} 
F(A'_{n+1}) \to G(A'_{n+1}) \times_{G(A'_n) \times_{G(A_n)} G(A_{n+1})}  ( F(A'_n) \times_{F(A_n)} F(A_{n+1}) ) 
\end{equation} 
is surjective. Let $(\Spec(A'_{n+1}) \to S, E'_n, E_{n+1})$ be a point of the target of this map, 
and set $E_n = (E_{n+1})_{A_n} \cong (E'_n)_{A_n}$. 
By Step~\ref{step4-smoothness}, the set $\Def_{E_{n+1}}(A'_{n+1})$  of deformations 
of $E_{n+1}$ over $A'_{n+1}$ is an $\Ext^1_{A_{n+1}}(E_{n+1}, E_{n+1})$-torsor, 
and the set $\Def_{E_{n}}(A'_n)$ of deformations of $E_n$ over $A'_n$ is an  
$\Ext^1_{A_n}(E_n, E_n)$-torsor.  
The restriction map 
\begin{equation}
\label{T1Def}
\Def_{E_{n+1}}(A'_{n+1}) \to \Def_{E_{n}}(A'_{n}) 
\end{equation} 
is compatible with the torsor structures under the natural map 
\begin{equation*}
\Ext^1_{{A_{n+1}}}(E_{n+1}, E_{n+1}) \to 
\Ext^1_{{A_n}}(E_n, E_n) . 
\end{equation*} 
By Lemma~\ref{lemma-simple-ug-CY2}\eqref{ExtEE-bc} the map on $\Ext^1$ 
groups is identified with the natural surjective map 
\begin{equation*}
\Ext^1_{{A_{n+1}}}(E_{n+1}, E_{n+1}) \to \Ext^1_{{A_{n+1}}}(E_{n+1}, E_{n+1}) \otimes_{A_{n+1}} A_n, 
\end{equation*} 
so~\eqref{T1Def} is also surjective. 
Thus there is an element $E'_{n+1} \in \Def_{E_{n+1}}(A'_{n+1})$ which restricts to $E'_n \in \Def_{E_{n}}(A'_{n})$. 
Equivalently, $E'_{n+1}$ maps to $(\Spec(A'_{n+1}) \to S, E'_n, E_{n+1})$ under~\eqref{T1-diagram}, which proves the required surjectivity. 
By Step~\ref{step3-smoothness}, this completes the proof of the theorem. 
\end{proof} 


\section{Proofs of results on the integral Hodge conjecture}
\label{section-theorem-proofs} 

In this section we prove Theorem~\ref{theorem-IHC}, Corollary~\ref{corollary-example-IHC}, and Theorem~\ref{theorem-VHC}
announced in the introduction, as well as several complementary results. 

We start with a general criterion for verifying 
the noncommutative variational integral Hodge conjecture. 

\begin{proposition}
\label{proposition-VHC-criterion}
Let $\cC \subset \Dperf(X)$ be an $S$-linear admissible subcategory, 
where $X \to S$ is a smooth proper morphism of complex varieties. 
Let $\vphi$ be a section of the local system $\Ktop[0](\cC/S)$. 
Assume there exists a complex point $0 \in S(\bC)$ such that the fiber $\vphi_0 \in \Ktop[0](\cC_0)$ is 
the class of an object $E_0 \in \cC_0$ with the property that 
$\cM(\cC/S, \vphi) \to S$ is smooth at $E_0$. 
Then $\vphi_s \in \Ktop[0](\cC_s)$ is algebraic for every $s \in S(\bC)$. 
\end{proposition}

\begin{proof}
Note that 
the conclusion of the proposition is insensitive to base change along a 
surjective morphism $S' \to S$. 
More precisely, given such an $S' \to S$, choose $0' \in S'(\bC)$ mapping to $0 \in S(\bC)$. 
By base change we obtain an $S'$-linear admissible 
subcategory $\cC' \subset \Dperf(X')$ where $X' = X \times_{S} S' \to S'$ and 
a section $\vphi'$ of $\Ktop[0](\cC'/S')$ such that $\vphi'_{0'}$ is the class of 
the pullback $E_{0'}$  of $E_0$ to $\cC_{0'}$, with the property that  
the morphism $\cM(\cC'/S', \vphi') \to S'$ is smooth at $E_{0'}$. 
Then the conclusion of the proposition for this base changed family implies the 
same result for the original family, since it is a fiberwise statement. 
We freely use this observation below. 

Let $\cM^{\circ} \subset \cM(\cC/S, \vphi)$ be the smooth locus of  
$\cM(\cC/S, \vphi) \to S$, i.e. the maximal open substack to which the restricted morphism 
is smooth (see \citestacks{0DZR}).  
By assumption $E_0$ lies in the fiber of $\cM^{\circ} \to S$ over $0$. 
The image of $\cM^{\circ} \to S$ is thus a nonempty open subset $U \subset S$ containing $0$. 
Since $\cM^{\circ} \to U$ is smooth and surjective, there exists a surjective \'{e}tale morphism 
$U' \to U$ with a point $0' \in U'(\bC)$ mapping to $0 \in U(\bC)$, such that 
the base change $\cM^{\circ}_{U'} \to U'$ admits a section taking $0'$ to $E_{0'} \in \cM^{\circ}_{U'}$. 
Thus, by base changing along a compactification $S' \to S$ of the morphism $U' \to S$, 
we may assume that $\cM^{\circ} \to U$ admits a section taking $0$ to $E_0$. 
In other words, if $\vphi_U$ denotes the section of $\Ktop[0](\cC_U/U)$ given by the restriction of $\vphi$, then there exists an object $E_U \in \cC_U$ of class $\vphi_U$ such that 
$(E_{U})_0 \simeq E_0$. 

Next we aim to produce a lift of $E_U \in \cC_U$ to an object $E \in \cC$. 
First, we may assume $S$ is smooth by base changing along a resolution of singularities. 
Then $X$ is smooth since it is smooth over $S$. 
It follows that the object $E_U \in \cC_U \subset \Dperf(X_U)$ lifts to an object 
$F \in \Dperf(X)$ (see e.g. \cite[Lemma 2.3.1]{polishchuk}). 
The projection of $F$ onto $\cC \subset \Dperf(X)$ then gives the desired lift $E \in \cC$ 
of $E_U$. 
The class $\vphi_E \in \Gamma(S^{\an}, \Ktop[0](\cC/S))$ of $E$ must equal $\vphi$, 
since these sections of the local system $\Ktop[0](\cC/S)$ agree over the open subset 
$U$. Therefore, $\vphi_s$ equals the class of $E_s$ for every $s \in S(\bC)$, and 
in particular is algebraic.  
\end{proof}

\begin{proof}[Proof of Theorem~\ref{theorem-VHC}] 
By Lemma~\ref{lemma-invariant-cycles} the fibers $\vphi_s \in  \tH(\cC_s, \bZ)$ 
are Hodge classes for all $s \in S(\bC)$, 
so by Theorem~\ref{theorem-mukai} the morphism $s\cM(\cC/S, \vphi) \to S$ is smooth. 
As the morphism $s\cM(\cC/S, \vphi) \to \cM(\cC/S)$ is an open immersion by Lemma~\ref{lemma-sM-lft}, 
it follows that the morphism $\cM(\cC/S, \vphi) \to S$ is smooth at any point of the domain corresponding to a simple universally gluable object. 
Thus Proposition~\ref{proposition-VHC-criterion} applies to show 
$\vphi_s \in \tH(\cC_s, \bZ)$ is algebraic for every $s \in S(\bC)$. 
\end{proof}

\begin{proof}[Proof of Theorem~\ref{theorem-IHC}] 
We may assume that $v$ is primitive, because if the result is true for $v$ then 
it is also true for any multiple of $v$. 
By Theorem~\ref{theorem-VHC}, it thus suffices to show that if $w \in \Hdg(\Dperf(T, \alpha), \bZ)$ is primitive and satisfies  
$(w,w) \geq -2$ or $(w,w) \geq 0$ according to whether $T$ is K3 or abelian, 
then it is the class of a simple universally gluable object in 
$\Dperf(T, \alpha)$. In fact, more is true: there is a nonempty distinguished component $\Stab^{\dagger}(T, \alpha)$ of the space of Bridgeland stability conditions on $\Dperf(T, \alpha)$, such that for $\sigma \in \Stab^{\dagger}(T, \alpha)$ generic with respect to $w$, the moduli space of $\sigma$-stable objects in $\Dperf(T, \alpha)$ of class $w$ is nonempty of dimension $(w,w)+2$. 
For $T$ a K3 surface this is \cite[Theorem 6.8]{bayer-macri} and \cite[Theorem 2.15]{bayer-macri-MMP} (based on \cite{yoshioka-abelian, yoshioka-twisted}), and 
for $T$ an abelian surface and $\alpha = 0$ this is \cite[Theorem 2.3]{bayer-li} (based on \cite{yoshioka-abelian16, MYY}) but the case of general $\alpha$ holds by 
similar arguments.  
This completes the proof, because a Bridgeland stable object is necessarily simple and universally gluable. 
\end{proof} 

Our proof of Corollary~\ref{corollary-example-IHC} will be based on the following result. 

\begin{proposition}
\label{proposition-Ku-IHC} 
Let $X$ be a cubic or GM fourfold. 
Then the integral Hodge conjecture holds for $\Ku(X)$. 
\end{proposition} 

\begin{proof}
We verify the criterion of Theorem~\ref{theorem-IHC}. 
First we claim that the cokernel of the map $\rK_0(\Ku(X)) \to \Hdg(\Ku(X), \bZ)$ is generated by elements $v \in \Hdg(\Ku(X), \bZ)$ with $(v, v) \geq -2$. 
Indeed, the image of $\rK_0(\Ku(X)) \to \Hdg(\Ku(X), \bZ)$ contains a class $\lambda$ with   
$(\lambda, \lambda) > 0$; in fact, the image contains --- and in the very general case equals --- a 
canonical positive definite rank $2$ sublattice, see \cite[\S2.4]{addington-thomas} and \cite[Lemma 2.27]{KuzPerry:dercatGM}. 
The claim then follows because for any $v \in \Hdg(\Ku(X), \bZ)$, we have $(v + t \lambda, v+t\lambda) \geq -2$ for $t$ a sufficiently large integer. 

Therefore, it suffices to show that for any $v \in \Hdg(\Ku(X), \bZ)$, 
there exists a family of cubic or GM fourfolds $Y \to S$ and  
points $0,1 \in S(\bC)$ such that 
$Y_0 \cong X$, 
$\Ku(Y_1) \simeq \Dperf(T, \alpha)$ for a twisted K3 surface $(T, \alpha)$, 
and $v$ remains of Hodge type along $S$. 

If $X$ is a cubic fourfold, the existence of such a family of cubic fourfolds $Y \to S$ 
follows from \cite[Theorem 4.1]{addington-thomas}. 
More precisely, the Kuznetsov component of a cubic fourfold containing a plane is equivalent 
to the derived of a twisted K3 surface \cite{kuznetsov-cubic}, 
and \cite[Theorem 4.1]{addington-thomas} uses 
Laza and Looijenga's description of the image of the period map for cubic fourfolds  \cite{laza, looijenga} 
to show that $X$ is deformation equivalent within the Hodge locus for $v$ to a cubic fourfold containing a plane. 

For GM fourfolds, the argument is more complicated, because less is known 
about the image of their period map. 
Recall a GM fourfold is called \emph{special} or \emph{ordinary} according to whether or not, 
in the notation of Definition~\ref{definition-GM}, the vertex of $\Cone(\Gr(2, V_5))$ is contained in the linear 
subspace $\bP^{8} \subset \bP^{10}$. 
By \cite[Theorem 1.2]{KuzPerry:dercatGM}, the Kuznetsov component of an ordinary GM fourfold containing a quintic del Pezzo surface is equivalent to the derived category of a K3 surface. 
Thus it suffices to show that if $X$ is a GM fourfold, it is deformation equivalent within the Hodge locus for $v$ to such a GM fourfold. 
If the conjectural description of the image of the period map for GM fourfolds were known  \cite[Question 9.1]{DIM4fold}, this could be proved analogously to 
the case of cubic fourfolds by a lattice theoretic computation. 
This conjecture is not known, but we can still use the period map to complete the argument as follows. 

By the construction of GM fourfolds in the proof of \cite[Theorem 8.1]{DIM4fold}, 
it follows that $X$ is deformation equivalent within the Hodge locus for $v$ to an 
ordinary GM fourfold $X'$ containing a so-called $\sigma$-plane.  
We claim that in the fiber through $X'$ of the period map for GM fourfolds, 
there is an ordinary GM fourfold $X''$ containing a quintic del Pezzo surface. 
Since preimages of irreducible subvarieties under 
the period map remain irreducible (see \cite[Lemma 5.12]{GMstability}), 
the claim implies that $X$ is deformation equivalent within the Hodge locus for $v$ to $X''$.  

To prove the claim, we freely use the notation and terminology on EPW sextics introduced in 
\cite[\S3]{DebKuz:birGM} and summarized in \cite[\S3]{KuzPerry:dercatGM}. 
Because $X'$ contains a $\sigma$-plane, by \cite[Remark 5.29]{DebKuz:periodGM} 
the EPW stratum $Y_{A(X')}^3 \subset \bP(V_6(X'))$ is nonempty. 
Let $\mathbf{p} \in Y_{A(X')^{\perp}}^{1} \subset \bP(V_6(X')^{\svee})$ be a point in the top stratum of the dual EPW sextic, 
such that the corresponding hyperplane in $\bP(V_6(X'))$ does not contain $Y_{A(X')}^3$. 
Let $X''$ be the ordinary GM fourfold corresponding to the pair $(A(X'), \mathbf{p})$ 
(see \cite[Theorem 3.10]{DebKuz:birGM} or \cite[Theorem~3.1]{KuzPerry:dercatGM}). 
Then \cite{DebKuz:periodGM} shows $X'$ and $X''$ lie in the same fiber of the period map, 
and \cite[Lemma 4.4]{KuzPerry:dercatGM} shows that $X''$ contains a quintic del Pezzo surface. 
This finishes the proof of the claim. 
\end{proof} 

\begin{remark}
\label{remark-CY2-IHC}
The result \cite[Proposition~5.8]{GMstability} gives the existence of families of GM fourfolds 
$Y \to S$ satisfying even stronger conditions than those required in the above proof. 
We preferred to give the above more elementary argument instead, 
because \cite[Proposition~5.8]{GMstability} relies on deep ingredients: 
the construction of stability conditions on Kuznetsov components of GM fourfolds, 
as well as the theory of stability conditions in families from \cite{BLMNPS}. 

In fact, one of the motivations for this paper was to develop a technique for proving 
the integral Hodge conjecture for CY2 categories that avoids the difficult 
problem of constructing stability conditions. 
As the proof of Proposition~\ref{proposition-Ku-IHC} illustrates, if our categories occur as the 
Kuznetsov components $\Ku(X) \subset \Dperf(X)$ of members $X$ of a family of varieties, 
our technique requires three ingredients: 
\begin{enumerate}
\item \label{IHC-ingredient-1}
The existence of a class $\lambda$ in the image of the map $\rK_0(\Ku(X)) \to \Hdg(\Ku(X), \bZ)$ 
which satisfies $(\lambda, \lambda) > 0$. 
\item \label{IHC-ingredient-2} 
The existence of $X$ such that $\Ku(X) \simeq \Dperf(T, \alpha)$ for a twisted K3 or abelian surface. 
\item \label{IHC-ingredient-3} 
Sufficient control of Hodge loci to ensure that they always contain $X$ as in~\eqref{IHC-ingredient-2}. 
\end{enumerate}
Condition~\eqref{IHC-ingredient-1} holds in all of the known examples of CY2 categories from~\S\ref{section-CY2-examples}, and we expect it holds whenever condition~\eqref{IHC-ingredient-2} does. 
In practice, checking conditions~\eqref{IHC-ingredient-2} and~\eqref{IHC-ingredient-3} requires more work, 
but there are many available tools, e.g. homological projective geometry \cite{kuznetsov-HPD, categorical-plucker, categorical-joins, categorical-cones} has been crucial in checking condition~\eqref{IHC-ingredient-2} in the known examples. 
\end{remark} 

\begin{proof}[Proof of Corollary~\ref{corollary-example-IHC}]
Let $X$ be a cubic or GM fourfold. 
We claim that $\rH^*(X, \bZ)$ is torsion free and $\rH^{2m}(X, \bZ)$ is of Tate type for $m > 2$. 
(In fact the Hodge diamond of $X$ can be computed explicitly, see \cite{hassett, DIM4fold}, but the following argument gives a simpler proof of the Tate type statement.) 
Indeed, for a cubic fourfold the claim holds by the Lefschetz hyperplane theorem. 
If $X$ is an ordinary GM fourfold, then projection from the vertex of $\Cone(\Gr(2,V_5))$ gives an isomorphism 
\begin{equation*}
X \cong \Gr(2, V_5) \cap \bP^8 \cap Q, 
\end{equation*} 
where $\bP^8 \subset \bP^9$ is a hyperplane in the Pl\"{u}cker space and $Q \subset \bP^8$ is a quadric hypersurface. In this case, the Lefschetz hyperplane theorem again gives the claim. 
This also implies the claim for special GM fourfolds, because they are deformation equivalent to ordinary GM fourfolds. 

By Proposition~\ref{proposition-IHC-vs-NCIHC}\eqref{NCIHC-implies-IHC}  it thus suffices to prove the integral Hodge conjecture for $\Dperf(X)$. 
Recall there is a semiorthogonal decomposition of $\Dperf(X)$ consisting of $\Ku(X)$ and copies of the derived category of a point. 
Therefore the result follows from Lemma~\ref{lemma-IHC-sod} and Proposition~\ref{proposition-Ku-IHC}. 
\end{proof} 

Similar arguments yield the following. 
\begin{corollary}
\label{corollary-GM6}
Let $X$ be a GM sixfold. 
Then the Voisin group $\Voi^3(X)$ is $2$-torsion. 
\end{corollary} 

\begin{proof}
By \cite{GM6-cohomology} the group $\rH^*(X, \bZ)$ is torsion free 
and $\rH^{2m}(X, \bZ)$ is of Tate type for $m > 3$. Thus, as in the proof of 
Corollary~\ref{corollary-example-IHC}, by Proposition~\ref{proposition-IHC-vs-NCIHC}\eqref{NCIHC-implies-IHC}
we reduce to proving the integral Hodge conjecture for $\Ku(X)$. 
By the duality conjecture for GM varieties \cite[Conjecture 3.7]{KuzPerry:dercatGM} proved in \cite[Theorem 1.6]{categorical-cones} and the description of generalized duals of GM varieties from \cite[Lemma 3.8]{KuzPerry:dercatGM}, there exists a GM fourfold $X'$ and an equivalence $\Ku(X) \simeq \Ku(X')$. 
Hence the result follows from Proposition~\ref{proposition-Ku-IHC}.  
\end{proof} 

\begin{remark}
It would be interesting to determine whether Corollary~\ref{corollary-GM6} is optimal, i.e. whether 
there exists a GM sixfold $X$ such that $\Voi^3(X) \neq 0$. 
\end{remark}


\newcommand{\etalchar}[1]{$^{#1}$}
\providecommand{\bysame}{\leavevmode\hbox to3em{\hrulefill}\thinspace}
\providecommand{\MR}{\relax\ifhmode\unskip\space\fi MR }
\providecommand{\MRhref}[2]{%
  \href{http://www.ams.org/mathscinet-getitem?mr=#1}{#2}
}
\providecommand{\href}[2]{#2}


\end{document}